\documentclass[a4paper, 12pt]{article}
\usepackage{amsmath, amsthm, amsfonts, enumerate, fullpage}
\usepackage[all]{xy}
\newtheorem*{thm*}{Main Theorem}

\newtheorem{thm}{Theorem}[section]
\newtheorem{lem}[thm]{Lemma}
\newtheorem{cor}[thm]{Corollary}
\newtheorem{con}[thm]{Conjecture}
\newtheorem{prop}[thm]{Proposition}

\newtheorem{obs}[thm]{Observation}
\newtheorem{ques}[thm]{Question}
\newtheorem{case}{Case}

\numberwithin{subcase}{case}
\newtheorem*{claim}{Claim}
\theoremstyle{definition}

\theoremstyle{remark}

\newenvironment{definition}[1][Definition.]{\begin{trivlist}
\item[\hskip \labelsep {\bfseries #1}]}{\end{trivlist}}

\begin{document}
\title{Connectedness and Hamiltonicity of graphs on vertex colorings}
\author{Daniel C. McDonald}
%\date{August 17, 2009}
\maketitle 
\setlength{\parskip}{1\baselineskip}
\begin{abstract}
Given a graph $H$, let $G^j_k(H)$ be the graph whose vertices are the proper $k$-colorings of $H$, with edges joining two colorings if $H$ contains a connected subgraph on at most $j$ vertices that includes all vertices where the colorings differ.  Properties of $G^1_k(H)$ have been investigated before, including connectedness (see \cite{CHJ}) and Hamiltonicity (see \cite{CM}).  We introduce and study the parameters $g_k(H)$ and $h_k(H)$, which denote the minimum $j$ such that $G^j_k(H)$ is connected or Hamiltonian, respectively.
\end{abstract}

\section{Introduction} 
\label{sec:grayintro}
For a positive integer $k$, set $[k]=\{1,\ldots,k\}$.  For a graph $H$, let $V(H)$ denote the vertex set of $H$, and let $E(H)$ denote the edge set of $H$.  A \emph{proper $k$-coloring} of $H$ is a function $\phi:V(H)\rightarrow [k]$ such that $\phi(u)\neq\phi(v)$ if $uv\in E(G)$.  If $H$ has a proper $k$-coloring, then we say that $H$ is \emph{$k$-colorable}.  The \emph{chromatic number} of $H$, denoted $\chi(H)$, is the least $k$ for which $H$ is $k$-colorable.

Suppose we have a proper $k$-coloring $\phi$ of a graph $H$, but we want to see what other proper $k$-colorings of $H$ look like.  We could generate such colorings by first coloring $H$ according to $\phi$ and then applying the following \emph{mixing process}: pick any vertex $v\in V(H)$, change the color on $v$ while maintaining a proper coloring (if possible), and repeat.  See Figure \ref{mixproc} for an example of the mixing process applied to a $3$-colorable graph $H$. 
\begin{figure}[h]
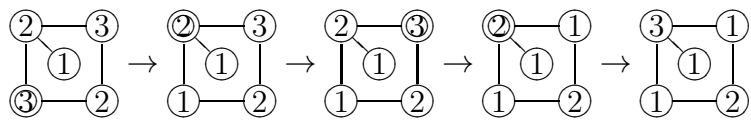

\centering
\[
\xygraph{ !{<0cm,0cm>;<.5cm,0cm>:<0cm,.5cm>::}
!{(0,0) }*+[o][F-:<3pt>]{1}="a"
!{(-1,1) }*+[o][F-:<3pt>]{2}="b"
!{(1,1) }*+[o][F-:<3pt>]{3}="c"
!{(1,-1) }*+[o][F-:<3pt>]{2}="d"
!{(-1,-1) }*+[o][F=:<3pt>]{3}="e"
"a"-"b"-"c"-"d"-"e"-"b"
}
\rightarrow
\xygraph{ !{<0cm,0cm>;<.5cm,0cm>:<0cm,.5cm>::}
!{(0,0) }*+[o][F-:<3pt>]{1}="a"
!{(-1,1) }*+[o][F=:<3pt>]{2}="b"
!{(1,1) }*+[o][F-:<3pt>]{3}="c"
!{(1,-1) }*+[o][F-:<3pt>]{2}="d"
!{(-1,-1) }*+[o][F-:<3pt>]{1}="e"
"a"-"b"-"c"-"d"-"e"-"b"
}
\rightarrow
\xygraph{ !{<0cm,0cm>;<.5cm,0cm>:<0cm,.5cm>::}
!{(0,0) }*+[o][F-:<3pt>]{1}="a"
!{(-1,1) }*+[o][F-:<3pt>]{2}="b"
!{(1,1) }*+[o][F=:<3pt>]{3}="c"
!{(1,-1) }*+[o][F-:<3pt>]{2}="d"
!{(-1,-1) }*+[o][F-:<3pt>]{1}="e"
"a"-"b"-"c"-"d"-"e"-"b"
}
\rightarrow
\xygraph{ !{<0cm,0cm>;<.5cm,0cm>:<0cm,.5cm>::}
!{(0,0) }*+[o][F-:<3pt>]{1}="a"
!{(-1,1) }*+[o][F=:<3pt>]{2}="b"
!{(1,1) }*+[o][F-:<3pt>]{1}="c"
!{(1,-1) }*+[o][F-:<3pt>]{2}="d"
!{(-1,-1) }*+[o][F-:<3pt>]{1}="e"
"a"-"b"-"c"-"d"-"e"-"b"
}
\rightarrow
\xygraph{ !{<0cm,0cm>;<.5cm,0cm>:<0cm,.5cm>::}
!{(0,0) }*+[o][F-:<3pt>]{1}="a"
!{(-1,1) }*+[o][F-:<3pt>]{3}="b"
!{(1,1) }*+[o][F-:<3pt>]{1}="c"
!{(1,-1) }*+[o][F-:<3pt>]{2}="d"
!{(-1,-1) }*+[o][F-:<3pt>]{1}="e"
"a"-"b"-"c"-"d"-"e"-"b"
}
\]
\caption{The mixing process.}
\label{mixproc}
\end{figure}

Let the \emph{$k$-color graph} of $H$, denoted $G_k(H)$, have the proper $k$-colorings of $H$ as its vertices, with two colorings adjacent whenever they differ on exactly one vertex.  We can obtain all proper $k$-colorings of $H$ using the mixing process if and only if $G_k(H)$ is connected.  The connectedness of $G_k(H)$ arises in the study of efficient algorithms for almost-uniform sampling of $k$-colorings; see \cite{J1} and \cite{J2}.  The \emph{mixing number} of $H$, denoted $k_1(H)$, is the least $K$ such that $G_k(H)$ is connected for all $k\geq K$.  In 2008, Cereceda, van den Heuvel, and Johnson \cite{CHJ} studied $k_1(H)$.  In particular, they showed that $k_1(H)\leq d+2$ if $H$ is \emph{$d$-degenerate}, meaning every subgraph of $H$ has a vertex of degree at most $d$.  Other papers on the connectedness of $G_k(H)$, or on finding paths between particular vertices of $G_k(H)$, include \cite{BC}, \cite{CHJ1}, and \cite{CHJ2}.

A \emph{Gray code} is an ordering of the elements of a given set such that consecutive elements differ in specified allowable small changes; a \emph{cyclic Gray code} is a Gray code where the elements are arranged in cyclic order.  Gray codes allow one to traverse an entire set of objects while doing little work changing between consecutive elements.  See \cite{S} for a survey on Gray codes.  A Gray code on the set of proper $k$-colorings of $H$ is an ordering of these colorings such that consecutive colorings differ on exactly one vertex.  There is a cyclic Gray code on the set of proper $k$-colorings of $H$ if and only if $G_k(H)$ is Hamiltonian. 

Cyclic Gray codes of proper colorings were first considered by Choo and MacGillivray \cite{CM} in 2011.  The \emph{Gray code number} of $H$, denoted $k_0(H)$, is the least $K$ such that $G_k(H)$ is Hamiltonian for all $k\geq K$.  Since every Hamiltonian graph is connected, we have $k_0(H)\geq k_1(H)$.  In \cite{CM} it was shown that $k_0(H)\leq d+3$ if $H$ is $d$-degenerate.

When $G_k(H)$ is not connected, but something similar to the mixing process is still desired, or when $G_k(H)$ is not Hamiltonian, but something similar to a cyclic Gray code of proper $k$-colorings of $H$ is desired, it is natural to ask by how much the adjacency conditions on $G_k(H)$ need to be relaxed.  We relax the requirement that consecutive colorings differ only on a single vertex, but we still want the differences between consecutive colorings to be localized.

\begin{definition}
For a graph $H$ and positive integer $k\geq\chi (H)$, let the \emph{$j$-localized $k$-coloring graph} of $H$, denoted $G^j_k(H)$, be the graph whose vertices are the proper $k$-colorings of $H$, with edges joining two colorings if $H$ contains a connected subgraph on at most $j$ vertices containing all vertices where the colorings differ (see Figure \ref{loccolgraphs}).  Let the \emph{$k$-color mixing number} of $H$, denoted $g_k(H)$, be the least $j$ such that $G^j_k(H)$ is connected, and let the \emph{$k$-color Gray code number} of $H$, denoted $h_k(H)$, be the least $j$ such that $G^j_k(H)$ is Hamiltonian.
\end{definition}

\begin{figure}[h]
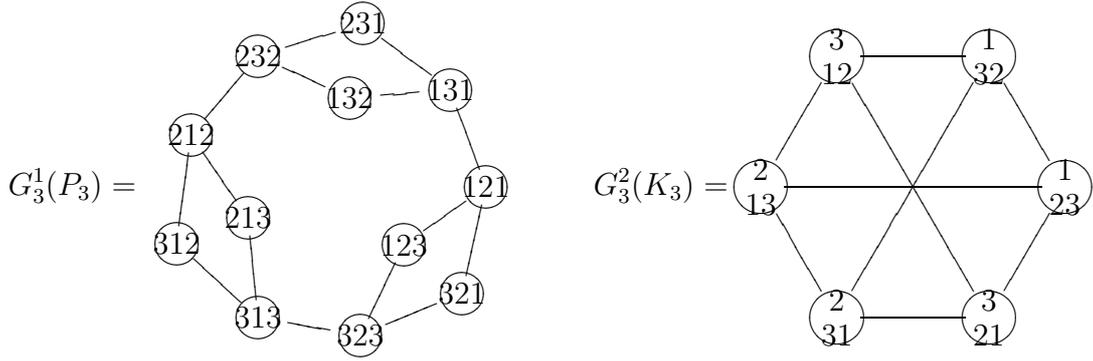

\centering
\[
G^1_3(P_3)= 
\xygraph{ !{<0cm,0cm>;<1cm,0cm>:<0cm,1cm>::}
!{(0,0);a(0)**{}?(2.0)}*+[o][F-:<3pt>]{121}="a"
!{(0,0);a(40)**{}?(2.0)}*+[o][F-:<3pt>]{131}="b"
!{(0,0);a(80)**{}?(1.2)}*+[o][F-:<3pt>]{132}="c"
!{(0,0);a(80)**{}?(2.2)}*+[o][F-:<3pt>]{231}="d"
!{(0,0);a(120)**{}?(2.0)}*+[o][F-:<3pt>]{232}="e"
!{(0,0);a(160)**{}?(2.0)}*+[o][F-:<3pt>]{212}="f"
!{(0,0);a(200)**{}?(1.2)}*+[o][F-:<3pt>]{213}="g"
!{(0,0);a(200)**{}?(2.2)}*+[o][F-:<3pt>]{312}="h"
!{(0,0);a(240)**{}?(2.0)}*+[o][F-:<3pt>]{313}="i"
!{(0,0);a(280)**{}?(2.0)}*+[o][F-:<3pt>]{323}="j"
!{(0,0);a(320)**{}?(1.2)}*+[o][F-:<3pt>]{123}="k"
!{(0,0);a(320)**{}?(2.2)}*+[o][F-:<3pt>]{321}="l"
"a"-"b"-"c"-"e"-"f"-"g"-"i"-"j"-"k"-"a"
"b"-"d"-"e"
"f"-"h"-"i"
"j"-"l"-"a"
}
\hspace{1cm}
G^2_3(K_3)= 
\xygraph{ !{<0cm,0cm>;<1cm,0cm>:<0cm,1cm>::}
!{(0,0);a(0)**{}?(2.0)}*+[o][F-:<3pt>]\txt{1\\ 23}="a"
!{(0,0);a(60)**{}?(2.0)}*+[o][F-:<3pt>]\txt{1\\ 32}="b"
!{(0,0);a(180)**{}?(2.0)}*+[o][F-:<3pt>]\txt{2\\ 13}="c"
!{(0,0);a(240)**{}?(2.0)}*+[o][F-:<3pt>]\txt{2\\ 31}="d"
!{(0,0);a(120)**{}?(2.0)}*+[o][F-:<3pt>]\txt{3\\ 12}="e"
!{(0,0);a(300)**{}?(2.0)}*+[o][F-:<3pt>]\txt{3\\ 21}="f"
"a"-"b"-"e"-"c"-"d"-"f"-"a"
"a"-"c"
"b"-"d"
"e"-"f"
}\]
\caption{Two examples of localized coloring graphs.}
\label{loccolgraphs}
\end{figure}

Since $G^1_k(H)=G_k(H)$, the statement ``$k_1(H)=K$'' is equivalent to ``$g_k(H)=1$ for $k\geq K$ but $g_{K-1}(H)>1$,'' and the statement ``$k_0(H)=K$'' is equivalent to ``$h_k(H)=1$ for $k\geq K$ but $h_{K-1}(H)>1$.''  Also note that if $j<\ell$, then $G^j_k(H)$ is a spanning subgraph of $G^{\ell}_k(H)$.  Clearly $g_k(H)\leq h_k(H)$, with $G^j_k(H)$ connected if and only if $j\geq g_k(H)$, and $G^j_k(H)$ Hamiltonian if and only if $j\geq h_k(H)$.

Rephrasing the previously stated degeneracy results, in \cite{CHJ} it is shown that $g_k(H)=1$ if $H$ is $(k-2)$-degenerate, and in \cite{CM} it is shown that $h_k(H)=1$ if $H$ is $(k-3)$-degenerate.  We first note that $g_k(H)$ and $h_k(H)$ exist whenever $k\geq\chi (H)$.  If $H$ is a connected $k$-colorable $n$-vertex graph, then $g_k(H)$ and $h_k(H)$ exist because $G^n_k(H)$ is a complete graph and thus Hamiltonian.  If $H$ consists of components $H_1,\ldots,H_m$, and $k\geq\chi (H)$, then clearly $G^j_k(H)=G^j_k(H_1)\square\cdots\square G^j_k(H_m)$.  The Cartesian product of graphs is connected if and only if each of the graphs is connected, and it is Hamiltonian if all are Hamiltonian, so $g_k(H)=\max_{i\in [m]}g_k(H_i)$ and $h_k(H)\leq \max_{i\in [m]}h_k(H_i)$ (see \cite{DPP} for details about the Hamiltonicity of Cartesian products).  

\begin{obs}
For every graph $H$ and integer $k\geq\chi (H)$, $g_k(H)$ and $h_k(H)$ exist.
\end{obs}

The inequality $h_k(H)\leq \max_{i\in [m]}h_k(H_i)$ is obviously an equality when $h_k(H_i)=1$ for each $i\in [m]$, but the inequality can also be strict: the Cartesian product of a Hamiltonian graph $G_1$ and a connected graph $G_2$ is Hamiltonian if the number of vertices of $G_1$ is at least the maximum degree of $G_2$, so if a graph $H$ has subgraphs $H'$ and $H''$ such that $H''=H-V(H')$ and there are no edges between $H'$ and $H''$ (that is, $H=H'+H''$), and the number of proper $k$-colorings of $H'$ is at least the maximum degree of $G^j_k(H'')$, where $j=g_k(H'')$, then $h_k(H)\leq\max\{h_k(H'),g_k(H'')\}$.  We construct such an $H$ by letting $H'$ be a set of at least two isolated vertices and letting $H''$ be the cycle $C_4$.  Note that $G^1_3(K_1)=K_3$, which is Hamiltonian, so $h_3(H')=1$ for all $n$.  Furthermore, there are $3^{|V(H')|}$ proper $3$-colorings of $H'$, and in \cite{CHJ} and \cite{CM} it is shown that $G^1_3(C_4)$ has maximum degree $4$ and is connected but not Hamiltonian, so $g_3(H'')=1<h_3(H'')$.  Thus $h_3(H)\leq\max\{h_3(H'),g_3(H'')\}=1<h_3(H'')$.

One would like to bound $g_k(H)$ and $h_k(H)$ in terms of $\chi (H)$ and $k$.  Such a statement is impossible, however: in Section \ref{sec:general} we generalize a construction from \cite{CHJ} to prove the following.

\begin{thm}\label{construction}
For $i$ and $k$ fixed with $1<i\leq k$, the functions $g_k$ and $h_k$ are unbounded on the set of $i$-chromatic graphs.  
\end{thm}

The construction $L_m$ from \cite{CHJ} is a bipartite graph such that $g_k(L_m)=1$ if and only if $3\leq k\neq m$; hence increasing $k$ can increase $g_k(H)$, though the degeneracy bounds imply that $g_k(H)=h_k(H)=1$ for large enough $k$.  The author has yet to see an example where increasing $k$ increases $h_k(H)$, though the construction from Theorem \ref{construction} would seem to be a promising candidate for such an $H$.

\begin{ques}
Does there exist a graph $H$ and integer $k$ such that $h_k(H)<h_{k+1}(H)$?
\end{ques}

In Section \ref{sec:subgraphs} we provide upper bounds for $g_k(H)$ and $h_k(H)$ in terms of $g_k(H')$ and $h_k(H')$ for certain induced subgraphs $H'$ of $H$.  The statements of these results involve the notion of \emph{choosability}.  Given a graph $F$ and function $f:V(F)\rightarrow\mathbb{N}$, an \emph{$f$-list assignment} for $F$ is a function $L$ that gives each $v\in V(F)$ a list of $f(v)$ positive integers.  An \emph{$L$-coloring} of $F$ is a proper coloring $\phi$ of $F$ such that $\phi(v)\in L(v)$ for all $v\in V(F)$.  A graph $F$ is $f$-choosable if every $f$-list assignment $L$ admits an $L$-coloring.  Note that if $F$ if $f$-choosable, then there exists a proper coloring $\phi$ of $F$ such that $\phi(v)\leq f(v)$ for each $v\in V(F)$ (simply let $\phi$ be an $L$-coloring for the $f$-list assignment $L$ defined by $L(v)=[f(v)]$ for all $v\in V(F)$).

As an application of the theorems of Section \ref{sec:subgraphs}, we consider $g_k(H)$ and $h_k(H)$ for any tree or cycle $H$.  In \cite{CHJ} it is shown that $g_3(C_n)=1$ if and only if $n=4$ (so $h_3(C_n)\geq g_3(C_n)>1$ for $n\neq 4$), and in \cite{CM} it is proved that $h_3(C_4)>1$ but $h_k(C_n)=1$ for $k\geq 4$ and $n\geq 3$ (so $g_k(C_n)=1$ for $k\geq 4$ and $n\geq 3$).  In \cite{CM} it is also proved for $k\geq 3$ and any tree $T$ that $h_k(T)=1$ except in the case $k=3$ and $T=K_{1,2m}$ for some $m\geq 1$ (so $g_3(T)=1$ if $T\neq K_{1,2m}$, and $h_3(K_{1,2m})>1$).  Obviously any connected $n$-vertex bipartite graph $H$ has exactly two proper $2$-colorings, which differ in all $n$ vertices, so $g_2(H)=h_2(H)=n$.  Since trees and cycles of even length are connected bipartite graphs, and cycles of odd length are not $2$-colorable, the only remaining computations for trees and cycles are $g_3(K_{1,2m})$, $h_3(K_{1,2m})$, $g_3(C_n)$, and $h_3(C_n)$.  We compute these values by applying the theorems of Section \ref{sec:subgraphs} and using the fact that if $H=K_{1,2m}$ or $H=C_n$ for $n\neq 4$, then there exists $v\in V(H)$ such that $H-v$ is some tree $T$ satisfying $h_3(T)=1$.

\begin{prop}\label{treecyclecol}
For $n\geq 3$, $g_3(C_n)=h_3(C_n)=2$ (except $g_3(C_4)=1$), and for $m\geq 1$, $g_3(K_{1,2m})=1$ and $h_3(K_{1,2m})=2$.
\end{prop}

If $\chi(F)>k\geq 2$ but we only have $k$ colors available, subdividing each edge of $F$ will alter $F$ into a $k$-colorable graph $H$ while still preserving some structure of $F$.  In Section \ref{sec:sub}, we bound $g_k(H)$ and $h_k(H)$ for $k\geq 3$ and and any graph $H$ obtained from a multigraph $M$ by subdividing each edge of $M$ at least some prescribed number of times (some edges can be subdivided more than others).  If $H$ can be constructed by subdividing each edge of $M$ once or more, then $H$ is $2$-degenerate, so $g_k(H)=1$ for $k\geq 4$ and $h_k(H)=1$ for $k\geq 5$.  We prove the following results.

\begin{thm}\label{subdivided}
Suppose that $H$ is obtained from a multigraph $M$ by subdividing each edge of $M$ at least $\ell$ times.  If $\ell=2$ and $M$ is loopless, then $g_3(H)\leq 2$ and $h_4(H)=1$.  If $\ell=3$, then $h_3(H)\leq 2$.
\end{thm}

Since $g_3(C_n)=2$ for $n\geq 4$, $k=4$ is the least number of colors for which $g_k(H)=1$ holds in general for graphs $H$ obtained from multigraphs $M$ by subdividing each edge of $M$ at least $\ell$ times for any $\ell$.  We believe the statements made about $h_k(H)$ in Theorem \ref{subdivided} can be improved, however.

\begin{con}
If $H$ is obtained from a multigraph $M$ by subdividing each edge of $M$ at least once, then $h_3(H)\leq 2$ and $h_4(H)=1$.
\end{con}

Many of the proofs in Sections \ref{sec:subgraphs} and \ref{sec:sub} follow the same pattern.  We are given a subgraph $H'$ of a graph $H$ such that every $k$-coloring of $H'$ can be extended to a $k$-coloring of $H$.  To compute an upper bound on $g_k(H)$ or $h_k(H)$ based on $g_k(H')$ or $h_k(H')$, we start with a path or Hamiltonian cycle in $G^j_k(H')$, and alter it into a path or Hamiltonian cycle in $G^{j'}_k(H)$ for some $j'$ not much larger than $j$.  In creating a Hamiltonian cycle in $G^{j'}_k(H)$, we list consecutively the extensions of each proper $k$-coloring of $H'$.  The surprisingly tricky aspect of such proofs is showing that we can close a Hamiltonian path through $G^{j'}_k(H)$ into a Hamiltonian cycle.

In \cite{CM} it is shown that $G^1_k(K_n)$ is edgeless if $k=n$ and Hamiltonian if $k>n$, so $h_n(K_n)\geq g_n(K_n)>1$ and $g_k(K_n)=h_k(K_n)=1$ for $k>n>1$.  Computing $g_n(K_n)$ and $h_n(K_n)$ is a matter of viewing proper $n$-colorings of $K_n$ as permutations on $[n]$ and listing them in cyclic order so that consecutive permutations differ only by transpositions (the oldest and most famous method for creating such a listing is the Steinhaus-Johnson-Trotter algorithm \cite{J}).  Hence $g_n(K_n)=h_n(K_n)=2$ for $n>1$.  In Section \ref{sec:multi} we use these results, plus one from Kompel'makher and Liskovets \cite{KL} on bases of transpositions, in generalizing from complete graphs to complete multipartite graphs.
\begin{thm}\label{multi}
If $H=K_{m_1,\ldots ,m_k}$, where $m_1\leq\cdots\leq m_k$, then the following hold: 
\begin{itemize}
\item
$g_k(H)=h_k(H)=m_1+m_k$
\item
$g_{\ell}(H)=1$ for $\ell>k$
\item
$h_{k+1}(H)=1$ if each $m_i$ is odd
\item
$h_{k+1}(H)=2$ if some $m_i$ is even
\end{itemize}
\end{thm}

We close this section by asking what relationships between $j$ and $k$ can guarantee the connectedness or Hamiltonicity of $G^j_k(H)$.  We have observed for a $d$-degenerate graph $H$ that $k\geq d+2$ implies $g_k(H)=1$ and $k\geq d+3$ implies $h_k(H)=1$, but the hypotheses of these statements are independent of $j$.  It would be interesting to see what functions $X(j,k)$ and $Y(H)$ nontrivially yield that $X(j,k)\geq Y(H)$ implies $g_k(H)\leq j$ or that $X(j,k)\geq Y(H)$ implies $h_k(H)\leq j$, potentially under restrictions of $j$, $k$, and $H$.  For example, we know $j+k\geq 4$ implies $h_k(C_n)\leq j$ for $j\geq 1$ and $k\geq 3$.  Continuing along these lines, we ask the following.

\begin{ques}
Are there constants $c$ and $c'$ such that if $H$ is $d$-degenerate, then $g_k(H)\leq j$ when $j\geq d-k-c$ and $h_k(H)\leq j$ when $j\geq d-k-c'$?
\end{ques}

\section{Unboundedness of $g_k$ and $h_k$ on Graphs with Fixed Chromatic Number}
\label{sec:general}

In this section we prove Theorem \ref{construction}, which states that for $i$ and $k$ fixed with $2\leq i\leq k$, the functions $g_k$ and $h_k$ are unbounded on the set of $i$-chromatic graphs.  It suffices to show that for $2\leq i\leq k$ and $j\geq 1$, there exists an $i$-chromatic graph $L(i,j,k)$ having a proper $k$-coloring $\phi$ isolated in $G^{j-1}_k(L(i,j,k))$.  

If $i=k$, then we can set $L(i,j,k)$ as the balanced complete $i$-partite graph with part size $\left\lceil j/2\right\rceil$.  Clearly $\chi (L(i,j,k))=i$, since $L(i,j,k)$ is an $i$-partite graph containing an $i$-clique.  Furthermore, we can let $\phi$ be any proper $k$-coloring of $L(i,j,k)$ since $G^{j-1}_k(L(i,j,k))$ is edgeless: any proper $k$-coloring of $L(i,j,k)$ assigns the colors of $[k]$ in a one-to-one fashion to the partite sets, each of which has at least $j/2$ vertices, so any distinct proper $k$-colorings of $L(i,j,k)$ differ on at least $j$ vertices.

For $i<k$, let $L(i,j,k)$ have $i$ partite sets each of size $k\left\lceil j/i\right\rceil$, with $\phi$ assigning each color in $[k]$ to exactly $\left\lceil j/i\right\rceil$ vertices in each partite set.  Give $L(i,j,k)$ precisely those edges that join differently colored vertices in different partite sets, so $\phi$ is a proper $k$-coloring of $L(i,j,k)$, and $\chi (L(i,j,k))=i$ since $L(i,j,k)$ is an $i$-partite graph containing an $i$-clique.  See Figure \ref{L(2,3,3)} for an illustration of $L(2,3,3)$ colored by $\phi$.  If $S$ is a set of $j-1$ vertices to be recolored, then $|X\cap S|<j/i$ for some partite set $X$.  For each $\ell\in[k]$ there is $x_{\ell}\in X-S$ such that $\phi(x_{\ell})=\ell$.  Thus recoloring any $y\in S-X$ with the color $\ell$ creates a monochromatic edge $x_{\ell}y$ in $L(i,j,k)$, so $\phi$ is isolated in $G^{j-1}_k(L(i,j,k))$.  

\begin{figure}[h]
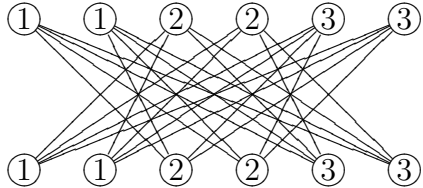

\centering
\[
\xygraph{ !{<0cm,0cm>;<.5cm,0cm>:<0cm,1cm>::}
!{(-5,1) }*+[o][F-:<3pt>]{1}="a1"
!{(-3,1) }*+[o][F-:<3pt>]{1}="a2"
!{(-1,1) }*+[o][F-:<3pt>]{2}="a3"
!{(1,1) }*+[o][F-:<3pt>]{2}="a4"
!{(3,1) }*+[o][F-:<3pt>]{3}="a5"
!{(5,1) }*+[o][F-:<3pt>]{3}="a6"
!{(-5,-1) }*+[o][F-:<3pt>]{1}="b1"
!{(-3,-1) }*+[o][F-:<3pt>]{1}="b2"
!{(-1,-1) }*+[o][F-:<3pt>]{2}="b3"
!{(1,-1) }*+[o][F-:<3pt>]{2}="b4"
!{(3,-1) }*+[o][F-:<3pt>]{3}="b5"
!{(5,-1) }*+[o][F-:<3pt>]{3}="b6"
"a1"-"b3" "a1"-"b4" "a1"-"b5" "a1"-"b6" 
"a2"-"b3" "a2"-"b4" "a2"-"b5" "a2"-"b6"
"a3"-"b1" "a3"-"b2" "a3"-"b5" "a3"-"b6"
"a4"-"b1" "a4"-"b2" "a4"-"b5" "a4"-"b6"
"a5"-"b1" "a5"-"b2" "a5"-"b3" "a5"-"b4"
"a6"-"b1" "a6"-"b2" "a6"-"b3" "a6"-"b4"
}\]
\caption{An illustration of $L(2,3,3)$ and its proper $3$-coloring $\phi$.}
\label{L(2,3,3)}
\end{figure}

The graph $L_m$ is defined in \cite{CHJ} as $K_{m,m}$ minus a perfect matching, and there it is shown for $k,m\geq 3$ that $G^1_k(L_m)$ is disconnected if and only if $k=m$.  We note that our $L(2,1,m)$ is identical to their $L_m$.

\section{Subgraphs}
\label{sec:subgraphs}

For this section, fix positive integers $j$ and $k$, a graph $H$, and disjoint subgraphs $H'$ and $H''$ of $H$ such that $\chi(H')\leq k$, $H''$ is connected and has at most $j$ vertices, and $H'=H-V(H'')$.  We investigate what can be said about $g_k(H)$ and $h_k(H)$ based on $g_k(H')$, $h_k(H')$, and $H''$.  Before continuing, we introduce some definitions to be used throughout the section.
\begin{definition}
Let $F$ be a subgraph of $H'$, and let $v\in V(H'')$.  Let $d^F(v)=|N_H(v)\cap V(F)|$; for convenience, set $d'(v)=d^{H'}(v)$.  Define $f^F(v)=k-d'(v)-d^F(v)$ and $f(v)=k-d'(v)-\min\{d'(v),j\}$.  For $u\in V(H'')$, define $f^F_u(v)=f^F(v)-\delta_{u,v}$ and $f_u(v)=f(v)-\delta_{u,v}$, where $\delta_{u,v}=1$ if $u=v$ and $\delta_{u,v}=0$ if $u\neq v$. 
\end{definition}
We start with the parameter $g_k(H)$, recalling the definition of choosability from Section \ref{sec:grayintro}.
\begin{prop}\label{gsubpropgen}
If $H''$ is $(k-d'(v))$-choosable, then $g_k(H)\leq g_k(H')+j$.
\end{prop}
\begin{proof}
Set $\ell=g_k(H')$.  Let $\phi$ and $\pi$ be any proper $k$-colorings of $H$, and let $\phi'$ and $\pi'$ be the proper $k$-colorings of $H'$ obtained, respectively, by restricting $\phi$ and $\pi$ to $H'$.  There exists a $(\phi',\pi')$-path in $G^{\ell}_k(H')$, which we alter into a $(\phi,\pi)$-path in $G^{\ell+j}_k(H)$ to complete the proof.  In $G^{\ell+j}_k(H)$, $\phi$ is adjacent to any other extension of $\phi'$ and $\pi$ is adjacent to any other extension of $\pi'$, so we need only show that any adjacent colorings $\alpha'$ and $\beta'$ in $G^{\ell}_k(H')$ have extensions $\alpha$ and $\beta$ that are adjacent in $G^{\ell+j}_k(H)$.  For $\gamma'\in\{\alpha',\beta'\}$, $\gamma'$ can be extended to a proper $k$-coloring of $H$ by coloring $H''$ from the list assignment $L$ defined by $L(v)=[k]-\{\gamma'(u):u\in V(H'),uv\in E(H)\}$, since $H''$ is $(k-d'(v))$-choosable.  

Let $F$ be a connected subgraph of $H'$ on at most $\ell$ vertices that includes everywhere $\alpha'$ and $\beta'$ differ.  If $H$ contains no edge joining $H''$ and $F$, then any coloring of $H''$ that extends $\alpha'$ to a proper $k$-coloring $\alpha$ of $H$ also extends $\beta'$ to a proper $k$-coloring $\beta$ of $H$, and $\alpha$ is adjacent to $\beta$ in $G^{\ell+j}_k(H)$ since they still only differ on $F$.  If some edge in $H$ joins $H''$ and $F$, then any extension $\alpha$ of $\alpha'$ is adjacent in $G^{\ell+j}_k(H)$ to any extension $\beta$ of $\beta'$, since they differ only on the subgraph of $H$ induced by $V(F)\cup V(H'')$, which is connected and has at most $\ell+j$ vertices.
\end{proof}
\begin{cor}\label{gsubprop}
If $H''$ consists of a single vertex $v$ having degree less than $k$ in $H$, then $g_k(H)\leq g_k(H')+1$.
\end{cor}
\begin{proof}
We have $k-d'(v)\geq 1$, so $H''$ is $(k-d'(v))$-choosable, so the result follows by setting $j=1$ in Proposition \ref{gsubpropgen}.
\end{proof}
Note that the hypothesis $k>d_H(v)$ is necessary in Corollary \ref{gsubprop}, since if $H'=K_k$ and $H=K_{k+1}$, then $H'$ is $k$-colorable but $H$ is not.
\begin{prop}\label{gchooseprop}
If $g_k(H')\leq j$ and $H''$ is $f^F$-choosable for each connected subgraph $F$ of $H'$ on at most $g_k(H')$ vertices, then $g_k(H)\leq j$.
\end{prop}
\begin{proof}
Let $\phi$ and $\pi$ be any proper $k$-colorings of $H$, and let $\phi'$ and $\pi'$ be the proper $k$-colorings of $H'$ obtained, respectively, by restricting $\phi$ and $\pi$ to $H'$.  There exists a $(\phi',\pi')$-path in $G^j_k(H')$, which we alter into a $(\phi,\pi)$-path in $G^j_k(H)$ to complete the proof.  If $\alpha'$ and $\beta'$ are adjacent colorings in $G^{j}_k(H')$, then the sets of extensions of $\alpha'$ and $\beta'$ to proper $k$-colorings of $H$ are cliques in $G^{j}_k(H)$, so we need only show that $\alpha'$ and $\beta'$ have extensions $\alpha$ and $\beta$ that are adjacent in $G^{j}_k(H)$.  

Let $F$ be a connected subgraph of $H'$ on at most $j$ vertices that includes everywhere $\alpha'$ and $\beta'$ differ.  Both $\alpha'$ and $\beta'$ can be extended to proper $k$-colorings $\alpha$ and $\beta$ of $H$ by coloring $H''$ from the list assignment $L$ defined by 
\begin{equation*}
L(v)=[k]-\{\alpha'(u):u\in V(H'),uv\in E(H)\}\cup\{\beta'(u):u\in V(H'),uv\in E(H)\}
\end{equation*} 
since $H''$ is $f^F$-choosable and 
\begin{equation*}
|\{\alpha'(u):u\in V(H'),uv\in E(H)\}\cup\{\beta'(u):u\in V(H'),uv\in E(H)\}|\leq d'(v)+d^F(v)
\end{equation*} 
for all $v\in V(H'')$.  Since $\alpha$ and $\beta$ only differ on $F$, they are adjacent in $G^{j}_k(H)$.
\end{proof}
\begin{cor}\label{gchoosecor}
If $g_k(H')\leq j$ and $H''$ is $f$-choosable, then $g_k(H)\leq j$.
\end{cor}
\begin{proof}
We need only show $f(v)\leq f^F(v)$ for any connected subgraph $F$ of $H'$ on at most $j$ vertices, since then $H''$ is $f^F$-choosable, and the result follows from Proposition \ref{gchooseprop}.  We have $d'(v)\geq d^F(v)$ since $V(F)\subseteq V(H')$, and $j\geq d^F(v)$ since $F$ has at most $j$ vertices, so $f^F(v)-f(v)=\min\{d'(v),j\}-d^F(v)\geq 0$.
\end{proof}
\begin{cor}\label{gvcor}
If $H''$ is a single vertex $v$ such that $k>d_{H}(v)+\min\{d_{H}(v),g_k(H')\}$, then $g_k(H)\leq g_k(H')$.
\end{cor}
\begin{proof}
Set $j=g_k(H')$ in Corollary \ref{gchoosecor}: $H''$ is $f$-choosable since $H''$ consists of a single vertex $v$ and 
\begin{equation*}
f(v)=k-d'(v)-\min\{d'(v),j\}=k-d_{H}(v)-\min\{d_{H}(v),g_k(H')\}\geq 1.
\end{equation*}
\end{proof}
\begin{cor}
Suppose $g_k(H')\geq 2$ and $H''$ is an edge $uv$.  If $d_H(v)\geq d_H(u)$ and $k\geq d_{H}(v)+\min\{d_{H}(v)-1,g_k(H')\}$, with at least one of these a strict inequality, then $g_k(H)\leq g_k(H')$.
\end{cor}
\begin{proof}
Set $j=g_k(H')$ in Corollary \ref{gchoosecor}: letting $d_H(v)=d_H(u)+\alpha$ and $k=d_{H}(v)+\min\{d_{H}(v)-1,g_k(H')\}+\beta$ (so $\alpha+\beta\geq 1$), $H''$ is $f$-choosable since $H''$ consists of an edge $uv$ and 
\begin{align*}
f(u) &=k-d'(u)-\min\{d'(u),j\} \\
&=k-d_{H}(u)+1-\min\{d_{H}(u)-1,g_k(H')\} \\
&\geq k-d_{H}(v)+\alpha+1-\min\{d_{H}(v)-1,g_k(H')\} \\
&=\alpha+\beta+1 \\
&\geq 2
\end{align*}
and
\begin{align*}
f(v) &=k-d'(v)-\min\{d'(v),j\} \\
&=k-d_{H}(v)+1-\min\{d_{H}(v)-1,g_k(H')\} \\
&=\beta+1 \\
&\geq 1.
\end{align*}
\end{proof}

We now turn to the parameter $h_k(H)$.
\begin{prop}\label{hvpropgen}
If $H''$ is $(k-d'(v))$-choosable, then $h_k(H)\leq h_k(H')+j$.
\end{prop}
\begin{proof}
We may assume that $H''$ is not its own component of $H$, since otherwise we would have $h_k(H)\leq\max\{h_k(H'),h_k(H'')\}\leq h_k(H')+j$.  Set $\ell=h_k(H')$, so there exists a Hamiltonian cycle $C'=[\phi^1,\ldots,\phi^b]$ through $G^{\ell}_k(H')$ such that $\phi^1$ and $\phi^b$ differ on a neighbor of a vertex in $H''$.  To complete the proof, we alter $C'$ into a Hamiltonian cycle $C$ through $G^{\ell+j}_k(H)$ such that the extensions of each $\phi^i$ appear consecutively in $C$, starting with $\alpha^i$ and ending with $\beta^i$.  Note that each $\phi^i$ can be extended to a proper $k$-coloring of $H$ by coloring $H''$ from the list assignment $L$ defined by $L(v)=[k]-\{\phi^i(u):u\in V(H'),uv\in E(H)\}$, since $H''$ is $(k-d'(v))$-choosable.  Thus the set of extensions of each $\phi^i$ to a proper $k$-coloring of $H$ is a nonempty clique in $G^{\ell+j}_k(H)$, so it suffices to order the extensions of each $\phi^i$ in any manner such that the last extension $\beta^i$ of $\phi^i$ is adjacent to the first extension $\alpha^{i+1}$ of $\phi^{i+1}$ in $G^{\ell+j}_k(H)$ (setting $b+1=1$).

Put the extensions of $\phi^1$ in any order, designating the first as $\alpha^1$ and the last as $\beta^1$.  Now consider $2<i\leq b$, and let $F$ be a connected subgraph of $H'$ on at most $\ell$ vertices that includes everywhere $\phi^{i-1}$ and $\phi^{i}$ differ.  If $H$ contains no edge joining $H''$ and $F$, then any coloring of $H''$ that extends $\phi^{i-1}$ to a proper $k$-coloring of $H$ also extends $\phi^{i}$ to a proper $k$-coloring of $H$, and these extensions are adjacent in $G^{\ell+j}_k(H)$ since they still only differ on $F$.  In this case, let $\alpha^i$ be any neighbor of $\beta^{i-1}$, and put the remaining extensions of $\phi^i$ in any order, designating the last as $\beta^i$.  If some edge in $H$ joins $H''$ and $F$, then any extension of $\phi^{i-1}$ is adjacent in $G^{\ell+j}_k(H)$ to any extension of $\phi^i$, since they differ only on the subgraph of $H$ induced by $V(F)\cup V(H'')$, which is connected and has at most $\ell+j$ vertices.  In this case, put the extensions of $\phi^i$ in any order.  Since we stipulated that $\phi^1$ and $\phi^b$ differ on a neighbor of a vertex in $H''$, this completes the Hamiltonian cycle $C$.
\end{proof}
\begin{cor}\label{hvprop}
If $H''$ consists of a single vertex $v$ having degree less than $k$ in $H$, then $h_k(H)\leq h_k(H')+1$.
\end{cor}
\begin{proof}
We have $k-d'(v)\geq 1$, so $H''$ is $(k-d'(v))$-choosable, so the result follows by setting $j=1$ in Proposition \ref{hvpropgen}.
\end{proof}

For distinct vertices $u$ and $v$ of $H''$ and a subgraph $F$ of $H'$, recall that $f^F_u(u)=f^F(u)-1$ and $f^F_u(v)=f^F(v)$.
\begin{lem}\label{adjext2}
Suppose $\phi$ and $\phi'$ are adjacent in $G^j_k(H')$, so the set of vertices on which $\phi$ and $\phi'$ differ lies in some connected subgraph $F$ of $H'$ on at most $j$ vertices.  If there exists $u\in V(H'')$ such that $H''$ is $f^F_u$-choosable, then there exist distinct proper $k$-colorings $\pi$ and $\rho$ of $H''$ each of which extends both $\phi$ and $\phi'$ to adjacent colorings in $G^j_k(H)$.
\end{lem}
\begin{proof}
For each $v\in V(H'')$, let $S(v)$ be the set of all colors used by $\phi$ and $\phi'$ on neighbors of $v$ in $H'$.  Define the list assignment $L$ for $H''$ by $L(v)=[k]-S(v)$, so any $L$-coloring of $H''$ extends $\phi$ and $\phi'$ to proper $k$-colorings of $H$ adjacent in $G^j_k(H)$ (since they would differ only on $F$).  To finish the proof, we use the fact that $H''$ is $f^F_u$-choosable to find distinct $L$-colorings $\pi$ and $\rho$ of $H''$.  Indeed, we can construct a $L$-coloring $\pi$ because, for all $v\in V(H'')$, 
\begin{equation*}
|L(v)|=k-|S(v)|\geq k-|N_H(v)\cap V(H')|-|N_H(v)\cap V(F)|=f^F(v)\leq f^F_u(v).
\end{equation*}
Now, obtain the $f^F_u$-list assignment $L'$ from $L$ by deleting $\pi(u)$ from $L(u)$.  We can find an $L'$-coloring $\rho$ because 
\begin{equation*}
|L'(u)|=|L(u)|-1\geq f^F(u)-1=f^F_u(u).
\end{equation*}
Since $\pi(u)\neq \rho(u)$ and $L'(v)\subseteq L(v)$ for each $v\in V(H'')$, $\pi$ and $\rho$ are distinct $L$-colorings.
\end{proof}
\begin{prop}\label{hchooseprop}
If $h_k(H')\leq j$, and for each connected subgraph $F$ of $H'$ on at most $j$ vertices, there exists $u\in V(H'')$ such that $H''$ is $f^F_u$-choosable, then $h_k(H)\leq j$.
\end{prop}
\begin{proof}
We may assume that $H''$ is not its own component of $H$, since otherwise we would have $h_k(H)\leq\max\{h_k(H'),h_k(H'')\}\leq j$ ($G^j_k(H'')$ is a complete graph since $H''$ is a connected graph on at most $j$ vertices, so $h_k(H'')=1$).  There exists a Hamiltonian cycle $C'=[\phi^1,\ldots,\phi^b]$ through $G^{j}_k(H')$; to complete the proof, we alter $C'$ into a Hamiltonian cycle $C$ through $G^{j}_k(H)$ such that the extensions of each $\phi^i$ appear consecutively in $C$, starting with $\alpha^i$ and ending with $\beta^i$.  By Lemma \ref{adjext2}, for each $i\in[b]$ there exist distinct proper $k$-colorings $\pi^i$ and $\rho^i$ of $H''$ each of which extend both $\phi^i$ and $\phi^{i-1}$ to adjacent colorings in $G^j_k(H)$.  Thus the set of extensions of each $\phi^i$ to a proper $k$-coloring of $H$ is a nonempty clique in $G^{j}_k(H)$, so it suffices to order the extensions of each $\phi^i$ in any manner such that the last extension $\beta^{i-1}$ of $\phi^{i-1}$ is adjacent to the first extension $\alpha^i$ of $\phi^{i}$ in $G^{j}_k(H)$ (setting $b+1=1$).

Certainly $\pi^1$ does not extend every proper $k$-coloring of $H'$ to a proper $k$-coloring of $H$ (by assumption some vertex $v$ in $H'$ neighbors a vertex in $H''$, and some proper $k$-coloring of $H'$ colors a neighbor of $v$ in $H'$ with $\pi^1(v)$).  Hence there exists $m\in[b-1]$ such that $\pi^1$ extends $\phi^1,\ldots,\phi^m$ to proper $k$-colorings of $H$, but $\pi^1$ does not extend $\phi^{m+1}$ to a proper $k$-coloring of $H$.  Let $\alpha^{m}$ be obtained from $\phi^m$ by coloring $H''$ according to $\pi^1$, and for $i\neq m$ let $\alpha^i$ be obtained from $\phi^i$ by coloring $H''$ according to whichever of $\pi^i$ or $\rho^i$ was not used for $\alpha^{i+1}$ (possibly neither $\pi^i$ nor $\rho^i$ was used for $\alpha^{i+1}$).  Thus for $i\in[b]$, $\alpha^i$ is adjacent in $G^j_k(H)$ to the extension $\beta^{i-1}$ of $\phi^{i-1}$ obtained by coloring $H''$ in the same way as $\alpha^i$, and $\beta^{i-1}\neq\alpha^{i-1}$ because they disagree on $H''$ ($\alpha^{m}\neq\alpha^{m+1}$ on $H''$ since $\phi^{m+1}$ cannot be extended to $H$ by coloring $H''$ according to $\pi^1$).  Put the other extensions of $\phi^{i}$ in any order between $\alpha^i$ and $\beta^i$.  This gives a Hamiltonian cycle through $G^j_k(H)$, since $\beta^{i-1}$ is adjacent to $\alpha^{i}$ in $G^{j}_k(H)$.
\end{proof}
For distinct vertices $u$ and $v$ of $H''$, recall that $f_u(u)=f(u)-1$ and $f_u(v)=f(v)$.  
\begin{cor}\label{hchoosecor}
If $h_k(H')\leq j$, and there exists $u\in V(H'')$ such that $H''$ is $f_u$-choosable, then $h_k(H)\leq j$.
\end{cor}
\begin{proof}
We need only show $f_u(v)\leq f^F_u(v)$ for each $u,v\in V(H'')$ and connected subgraph $F$ of $H'$ on at most $j$ vertices, since then $H''$ would be $f^F_u$-choosable, and the result would follow from Proposition \ref{hchooseprop}.  We have $d'(v)\geq d^F(v)$ since $V(F)\subseteq V(H')$, and $j\geq d^F(v)$ since $F$ has at most $j$ vertices, so $f^F_u(v)-f_u(v)=\min\{d'(v),j\}-d^F(v)\geq 0$.
\end{proof}
\begin{cor}\label{hvcor}
If $H''$ is a single vertex $u$ and $k\geq 2+d_{H}(u)+\min\{d_{H}(u),h_k(H')\}$, then $h_k(H)\leq h_k(H')$.
\end{cor}
\begin{proof}
Set $j=h_k(H')$ in Corollary \ref{hchoosecor}: $H''$ is $f_u$-choosable since $H''$ consists of a single vertex $u$ and 
\begin{equation*}
f_u(u)=f(u)-1=k-d'(u)-\min\{d'(u),j\}-1=k-d_{H}(u)-\min\{d_{H}(u),g_k(H')\}-1\geq 1.
\end{equation*}
\end{proof}
\begin{cor}
Suppose $h_k(H')\geq 2$ and $H''$ is a single edge $uv$.  If $d_H(v)\geq d_H(u)$ and $k\geq d_{H}(v)+\min\{d_{H}(v)-1,h_k(H')\}+1$, or if $d_H(v)\geq d_H(u)+2$ and $k\geq d_{H}(v)+\min\{d_{H}(v)-1,h_k(H')\}$, then $h_k(H)\leq h_k(H')$.
\end{cor}
\begin{proof}
Set $j=h_k(H')$ in Corollary \ref{hchoosecor}: letting $d_H(v)=d_H(u)+\alpha$ and $k=d_{H}(v)+\min\{d_{H}(v)-1,h_k(H')\}+\beta$ (so $\beta\geq 1$ or $\alpha\geq 2$), we have 
\begin{align*}
f_u(u) &=k-d'(u)-1-\min\{d'(u),j\} \\
&=k-d_{H}(u)-\min\{d_{H}(u)-1,h_k(H')\} \\
&\geq k-d_H(v)+\alpha-\min\{d_H(v)-1,h_k(H')\} \\
&=\alpha+\beta
\end{align*}  
and
\begin{align*}
f_u(v) &=k-d'(v)-\min\{d'(v),j\} \\
&=k-d_{H}(v)+1-\min\{d_{H}(v)-1,h_k(H')\} \\
&=\beta+1.
\end{align*}
If $\beta\geq 1$, then $f_u(u)\geq 1$ and $f_u(v)\geq 2$, and if $\alpha\geq 2$, then $f_u(u)\geq 2$ and $f_u(v)\geq 1$.  Hence $H''$ is $f_u$-choosable, so $h_k(H)\leq j=h_k(H')$.
\end{proof}
We note that Corollaries \ref{gvcor} and \ref{hvcor} can be used to recover the results in \cite{CHJ} and \cite{CM} that respectively state $g_k(H)=1$ if $H$ is $(k-2)$-degenerate, and $h_k(H)=1$ if $H$ is $(k-3)$-degenerate.  Indeed, order $V(H)$ as $v_1,\ldots,v_n$, where $v_n$ is a vertex of minimum degree in $H$, and for each $i\in [n-1]$, $v_i$ is a vertex of minimum degree in the induced subgraph $H_i$ of $H$ defined by $H_i=H-\{v_{i+1},\ldots,v_n\}$.  Setting $H_n=H$, we have $d_{H_i}(v_i)\leq d$ for $i\in[n]$ if $H$ is $d$-degenerate.  If $k=d+2$, then clearly $g_k(H_1)=1$ ($G^1_k(H_1)$ is a complete graph on $k$ vertices), and if $g_k(H_{i-1})=1$, then we get $g_k(H_i)=1$ by Corollary \ref{gvcor}, since $k=d+2>d_{H_i}(v_i)+1= d_{H_i}(v_i)+\min\{d_{H_i}(v_i),g_k(H_{i-1})\}$.  If $k=d+3$, then clearly $h_k(H_1)=1$ ($G^1_k(H_1)$ is a complete graph on $k$ vertices for some $k\geq 3$), and if $h_k(H_{i-1})=1$, then we get $h_k(H_i)=1$ by Corollary \ref{hvcor}, since $k=d+3\geq 2+d_{H_i}(v_i)+1= 2+d_{H_i}(v_i)+\min\{d_{H_i}(v_i),h_k(H_{i-1})\}$.

To conclude this section, we prove Proposition \ref{treecyclecol} concerning the computations $g_3(K_{1,2m})$, $h_3(K_{1,2m})$, $g_3(C_n)$, and $h_3(C_n)$.  

Setting $H=K_{1,2m}$ and $H''=v$ for some leaf $v$ of $H$, we have $d_H(v)=1$ and $H'=K_{1,2m-1}$.  Hence $g_3(H')=h_3(H')=1$, so $g_3(K_{1,2m})=1$, by Corollary \ref{gvcor}, and $h_3(K_{1,2m})=2$, by Corollary \ref{hvprop} (and the fact that $h_3(K_{1,2m})>1$).  

Setting $H=C_n$ for $n\neq 4$ and $H''=v$ for any vertex $v$ of $H$, we have $d_H(v)=2$ and $H'=P_{n-1}$.  Hence $h_3(H')=1$ since $n\neq 4$, so $h_3(C_n)=2$, by Corollary \ref{hvprop} (and the fact that $h_3(C_n)>1$).

Finally, we confirm $h_3(C_4)=2$ by exhibiting the following Hamiltonian cycle through $G^2_3(C_4)$: 1312, 1212, 1232, 1213, 1313, 1323, 2123, 2323, 2313, 2321, 2121, 2131, 3231, 3131, 3121, 3132, 3232, 3212.

\section{Subdividing Edges}
\label{sec:sub}
In this section, we prove Theorem \ref{multi}, concerning a graph $H$ obtained from a multigraph $M$ by subdividing each edge of $M$ at least $\ell$ times for some $\ell\geq 1$ (different edges need not receive the same number of subdivisions).  Note that $\chi(H)\leq 3$: the vertices of $H$ that originated in $M$ form an independent set in $H$ and thus can each be given color $1$, and a proper $3$-coloring of $H$ can be completed by coloring the remaining vertices from $\{2,3\}$ since each component of $H-V(M)$ is a path.  

For an induced subgraph $H'$ of a graph $H$, write $H'\subset^{x,y}_{\ell} H$ if $H-V(H')$ consists of a path $v_1,\ldots,v_{\ell}$ such that $d_H(v_j)=2$ for $j\in[\ell]$, with $v_1$ adjacent to $x\in V(H')$ and $v_{\ell}$ adjacent to $y\in V(H')$ (potentially $x=y$).  See Figure \ref{g3fig}.  Note that if $k\geq 3$, $\ell\geq 1$, and $H'\subset^{x,y}_{\ell} H$, then every proper $k$-coloring of $H'$ can be extended to a proper $k$-coloring of $H$.

Obtain a subforest $F$ of $H$ by deleting $\ell$ consecutive vertices from each subdivision of an edge in $M$.  Note that each component of $H-V(F)$ is a path on $\ell$ vertices $v_1,\ldots,v_\ell$ such that $d_H(v_j)=2$ for $j\in[\ell]$.  By adding these components of $H-V(F)$ back to $F$ one at a time, we get the following observation.
\begin{obs}\label{subobs}
If a graph $H$ is obtained from a multigraph $M$ by subdividing each edge of $M$ at least $\ell$ times, then there exists a sequence $H_1,\ldots,H_m$ of subgraphs of $H$ such that $H_1$ is a forest, $H_m=H$, and for $i\in[m-1]$ there exist $x_i,y_i\in V(H_i)$ such that $H_i\subset^{x_i,y_i}_{\ell}H_{i+1}$ and either $x_i=y_i$ or $d_{H_i}(x_i,y_i)>\ell$ (always the case if $M$ is loopless).
\end{obs}
\begin{figure}[h]
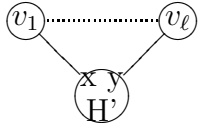

\centering
\[
\xygraph{ !{<0cm,0cm>;<1cm,0cm>:<0cm,1cm>::}
!{(0,0) }*+[o][F-:<3pt>]\txt{x y\\ H'}="a"
!{(-1,1) }*+[o][F-:<3pt>]{v_1}="b"
!{(1,1) }*+[o][F-:<3pt>]{v_{\ell}}="c"
"a"-"b"-@{.}"c"-"a"
}
\]
\caption{Graphs $H$ and $H'$ such that $H'\subset^{x,y}_{\ell} H$.}
\label{g3fig}
\end{figure}
\begin{prop}\label{g3subdivide}
Let $H'$ be a $3$-colorable subgraph of a graph $H$ such that $H-V(H')$ consists of an edge $uv$, with $u$ having a single neighbor $x\in V(H')$ and $v$ having a single neighbor $y\in V(H')-N[x]$.  If $g_3(H')\leq 2$, then $g_3(H)\leq 2$.
\end{prop}
\begin{proof}
Set $H''$ as the edge $uv$, $j=2$, and $k=3$ in Proposition \ref{gchooseprop} (if $F$ is a connected subgraph of $H'$ on at most $2$ vertices, then $F$ does not contain both $x$ and $y$ since $y\notin N[x]$, so either $f^F(u)\geq 1$ and $f^F(v)\geq 2$ or $f^F(u)\geq 2$ and $f^F(v)\geq 1$; either way $H''$ is $f^F$-choosable).
\end{proof}
\begin{cor}\label{g3subcor}
If $H$ is obtained from a loopless multigraph $M$ by subdividing each edge of $M$ at least twice, then $g_3(H)\leq 2$.
\end{cor}
\begin{proof}
Let $H_1,\ldots,H_m$ be a sequence of subgraphs of $H$ described in Observation \ref{subobs}; since $M$ is loopless, for each $i\in[m-1]$ we have $H_i\subset^{x_i,y_i}_{\ell}H_{i+1}$ where $d_{H_i}(x_i,y_i)\geq 3$.  We have $g_3(H_1)=1$ since $H_1$ is a forest, and for $i\in [m-1]$, if $g_3(H_i)\leq 2$, then $g_3(H_{i+1})\leq 2$, by Proposition \ref{g3subdivide}.  Hence $g_3(H)\leq 2$.
\end{proof}
We note that the condition that $M$ be loopless is necessary for Corollary \ref{g3subcor} to hold.  Indeed, suppose $M$ has a vertex $x$ with loops $L_1,\ldots,L_j$ that are subdivided exactly twice in forming $H$, with new vertices $u_i$ and $v_i$ in $L_i$ for $i\in[j]$.  If $\phi$ and $\phi'$ are proper $3$-colorings of $H$ such that $\phi(x)\neq \phi'(x)$, then $\phi$ and $\phi'$ lie in different components of $G^j_3(H)$: for each $i\in[j]$, $u_i$ and $v_i$ are neighbors of $x$, and $\{\phi(u_i),\phi(v_i)\}=[3]-\{\phi(x)\}$ since $\phi$ is proper, so $x$ cannot be recolored without also recoloring one of the new vertices from each of $L_1,\ldots,L_j$.

Let $H'$ be a $4$-colorable subgraph of a graph $H$ such that $H-V(H')$ consists of an edge $uv$, with $u$ having a single neighbor $x\in V(H')$ and $v$ having a single neighbor $y\in V(H')-\{x\}$.  For proper $4$-colorings $\psi^1$ and $\psi^2$ of $H'$ satisfying $\psi^1(x)=\psi^2(x)=1$ and $\psi^i(y)=i$, Figure \ref{adjexthfourfig} shows each subgraph $F^i$ of $G^1_4(H)$ induced by the set of proper $4$-colorings $\psi^i_1,\psi^i_2,\ldots$ of $H$ that agree with $\psi^i$ on $H'$, with node $\psi^i_{\ell}$ of $F^i$ labeled $\psi^i_{\ell}(x)\psi^i_{\ell}(u)\psi^i_{\ell}(v)\psi^i_{\ell}(y)$.  Note that if $\pi$ is one of the vertices of $F^2$ labeled $1212$, $1342$, or $1432$, and $\alpha$ is any vertex of $F^2$ besides $\pi$, then there is a Hamiltonian path through $F^2$ whose endpoints are $\pi$ and $\alpha$.  If instead $\pi$ is in $\{1232,1412\}$ but $\alpha$ is not, or $\pi$ is in $\{1242,1312\}$ but $\alpha$ is not, then again there is a Hamiltonian path through $F^2$ whose endpoints are $\pi$ and $\alpha$.
\begin{figure}[h]
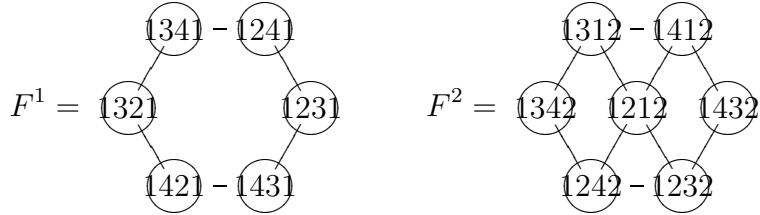

\centering
\[
F^1=
\xygraph{ !{<0cm,0cm>;<1cm,0cm>:<0cm,1cm>::}
!{(0,0);a(0)**{}?(1.2)}*+[o][F-:<3pt>]{1231}="a"
!{(0,0);a(60)**{}?(1.2)}*+[o][F-:<3pt>]{1241}="b"
!{(0,0);a(120)**{}?(1.2)}*+[o][F-:<3pt>]{1341}="c"
!{(0,0);a(180)**{}?(1.2)}*+[o][F-:<3pt>]{1321}="d"
!{(0,0);a(240)**{}?(1.2)}*+[o][F-:<3pt>]{1421}="e"
!{(0,0);a(300)**{}?(1.2)}*+[o][F-:<3pt>]{1431}="f"
"a"-"b"-"c"-"d"-"e"-"f"-"a"
}
\hspace{1cm}
F^2=
\xygraph{ !{<0cm,0cm>;<1cm,0cm>:<0cm,1cm>::}
!{(0,0);a(0)**{}?(0.0)}*+[o][F-:<3pt>]{1212}="a"
!{(0,0);a(120)**{}?(1.2)}*+[o][F-:<3pt>]{1312}="b"
!{(0,0);a(180)**{}?(1.2)}*+[o][F-:<3pt>]{1342}="c"
!{(0,0);a(240)**{}?(1.2)}*+[o][F-:<3pt>]{1242}="d"
!{(0,0);a(300)**{}?(1.2)}*+[o][F-:<3pt>]{1232}="e"
!{(0,0);a(0)**{}?(1.2)}*+[o][F-:<3pt>]{1432}="f"
!{(0,0);a(60)**{}?(1.2)}*+[o][F-:<3pt>]{1412}="g"
"g"-"a"-"b"-"c"-"d"-"e"-"f"-"g"-"b"
"d"-"a"-"e"
}\]
\caption{Two induced subgraphs of $G^1_4(H)$.}
\label{adjexthfourfig}
\end{figure}
\begin{lem}\label{adjexth4}
Let $H'$ be a $4$-colorable subgraph of a graph $H$ such that $H-V(H')$ consists of an edge $uv$, with $u$ having a single neighbor $x\in V(H')$ and $v$ having a single neighbor $y\in V(H')-\{x\}$, and let $\phi$ and $\phi'$ be proper $4$-colorings of $H'$ adjacent in $G^1_4(H')$.  Letting $G$ denote the subgraph of $G^1_4(H)$ induced by the extensions of $\phi$, for every $\pi\in V(G)$ there exists $\rho\in V(G)-\{\pi\}$ such that there is a Hamiltonian path through $G$ from $\pi$ to $\rho$, and $\rho$ is adjacent in $G^1_4(H)$ to some extension of $\phi'$.
\end{lem}
\begin{proof}
Since $\phi$ and $\phi'$ are adjacent in $G^1_4(H')$, they differ on exactly one vertex $w$ of $H'$.  Thus we may assume without loss of generality that $\phi(x)=\phi'(x)=1$.  Let $\pi\in V(G)$; we find $\rho\in V(G)-\{\pi\}$ such that there is a Hamiltonian path through $G$ from $\pi$ to $\rho$, with $\rho(u)\neq \phi'(x)$ and $\rho(v)\neq \phi'(y)$ (allowing $\phi'$ to be extended to some proper $k$-coloring $\rho'$ of $H$ by coloring $uv$ like $\rho$ does, so $\rho$ and $\rho'$ will be adjacent in $G^1_4(H)$ since they only differ on $w$).

First suppose $\phi(y)=1$, in which case $G$ looks like $F^1$ from Figure \ref{adjexthfourfig}.  Either $\phi'(y)=1$ or $\phi'(y)\neq 1$, in which case without loss of generality assume $\phi'(y)=2$.  In either case, there are extensions of both $\phi$ and $\phi'$ to $H$ that label $uv$ as $43$, $23$, $24$, and $34$, with every vertex in $G$ adjacent to at least one of these extensions.  Thus no matter whether $\phi'(y)=1$ or $\phi'(y)=2$, we can let $\rho$ be a neighbor of $\pi$ that labels $uv$ as either $43$, $23$, $24$, or $34$ ($\rho$ ends the Hamiltonian path through $G$ that starts at $\pi$ and moves in the opposite direction from $\rho$).

Now suppose $\phi(y)\neq 1$; without loss of generality assume $\phi(y)=2$, in which case $G$ looks like $F^2$ from Figure \ref{adjexthfourfig}.  If $\phi'(y)\in[2]$, then there are extensions of both $\phi$ and $\phi'$ that label $uv$ as $43$ and $34$; for each $\pi\in V(G)$ there is a Hamiltonian path through $G$ from $\pi$ to at least one of these vertices, which we set as $\rho$.  If $\phi'(y)\notin[2]$, then we assume without loss of generality that $\phi'(y)=3$, in which case there are extensions of both $\phi$ and $\phi'$ that label $uv$ as $24$ and $41$; for each $\pi\in V(G)$ there is a Hamiltonian path through $G$ from $\pi$ to at least one of these vertices, which we set as $\rho$.
\end{proof}
\begin{prop}\label{h4subdivide}
Let $H'$ be a $4$-colorable subgraph of a graph $H$ such that $H-V(H')$ consists of an edge $uv$, with $u$ having a single neighbor $x\in V(H')$ and $v$ having a single neighbor $y\in V(H')-\{x\}$.  If $h_4(H')=1$, then $h_4(H)=1$.
\end{prop}
\begin{proof}
Since $h_4(P_4)=1$, we may assume there exists a vertex $z\in V(H')-\{x,y\}$.  Since $h_4(H')=1$, there exists a Hamiltonian cycle $[\phi^1,\ldots,\phi^b]$ through $G^1_4(H')$.  There exists $i$ such that $\phi^i(z)\neq \phi^{i+1}(z)$, in which case $\phi^i(x)=\phi^{i+1}(x)$ and $\phi^i(y)=\phi^{i+1}(y)$.  If there exists an $i$ such that $\phi^i(x)=\phi^{i+1}(x)\neq \phi^i(y)=\phi^{i+1}(y)$, then without loss of generality assume $\phi^{b-1}(x)=\phi^b(x)=1$ and $\phi^{b-1}(y)=\phi^b(y)=2$.  If there exists no such $i$, then there must exist $\ell$ such that $\phi^{\ell}(x)=\phi^{\ell}(y)=\phi^{\ell+1}(x)=\phi^{\ell+1}(y)$, but either $\phi^{\ell+1}(x)=\phi^{\ell+2}(x)\neq \phi^{\ell+2}(y)$ or $\phi^{\ell+1}(y)=\phi^{\ell+2}(y)\neq \phi^{\ell+2}(x)$; without loss of generality assume $\phi^{b-2}(x)=\phi^{b-2}(y)=\phi^{b-1}(x)=\phi^{b-1}(y)=\phi^b(x)=1$ and $\phi^b(y)=2$.  Call this situation Case 1, and call the previously discussed situation Case 2.  To complete the proof, we alter the Hamiltonian cycle through $G^1_4(H')$ into a Hamiltonian cycle through $G^1_4(H)$ such that the extensions of each $\phi^i$ appear consecutively, starting with $\alpha^i$ and ending with $\beta^i$, with $\beta^i$ agreeing with $\alpha^{i+1}$ on $u$ and $v$.

For each $i\in[b]$, let $G^i$ denote the subgraph of $G^1_4(H)$ induced by the extensions of $\phi^i$. By Lemma \ref{adjexth4}, for every $\pi\in V(G^i)$ there exists $\rho\in V(G)-\{\pi\}$ such that there is a Hamiltonian path through $G^i$ from $\pi$ to $\rho$, and $\rho$ is adjacent in $G^1_4(H)$ to some extension of $\phi^{i+1}$.  Let $\alpha^1$ be any coloring in $V(G^1)$ for which there exist distinct colorings $\pi$ and $\rho$ in $V(G^b)$ such that there is a Hamiltonian path through $G^b$ from $\pi$ to $\rho$, and $\rho$ is adjacent in $G^1_4(H)$ to $\alpha^1$.  Letting $m=b-2$ if Case 1 holds and $m=b-3$ if Case 2 holds, order the extensions of $\phi^1,\ldots,\phi^m$, plus $\phi^{m+1}_0$, so that the extensions of each $\phi^i$ form a Hamiltonian path through $G^i$ from $\alpha^i$ to $\beta^i$, with $\beta^i$ adjacent in $G^1_4(H)$ to $\alpha^{i+1}$.  Let $\beta^b$ be the coloring in $V(G^b)$ adjacent in $G^1_4(H)$ to $\alpha^1$.
\begin{case}
We have $m=b-3$ as well as $\phi^{b-2}(x)=\phi^{b-2}(y)=\phi^{b-1}(x)=\phi^{b-1}(y)=\phi^b(x)=1$ and $\phi^b(y)=2$.
\end{case}
Note that $G^{b-2}$ and $G^{b-1}$ both look like $F^1$ from Figure \ref{adjexthfourfig}, while $G^b$ looks like $F^2$.  If we select $\beta^{b-2}$ as a neighbor of $\alpha^{b-2}$ in $G^{b-2}$, $\alpha^{b-1}$ as the coloring in $V(G^{b-1})$ that agrees with $\beta^{b-2}$ on $u$ and $v$, and $\beta^{b-1}$ as some neighbor in $G^{b-1}$ of $\alpha^{b-1}$, then there is a path in $G^1_4(H)$ that first touches every vertex of $G^{b-2}$ and then every vertex of $G^{b-1}$ ($\beta^{b-2}$ is adjacent in $G^1_4(H)$ to $\alpha^{b-1}$ because they only differ on the vertex of $H'$ where $\phi^{b-2}$ and $\phi^{b-1}$ differ).  

If $\beta^b$ uses a color outside of $\{3,4\}$ on $u$ or $v$, then set $\beta^{b-2}$ as a common neighbor in $G^{b-2}$ of $\alpha^{b-2}$ and a coloring $\pi\in V(G^{b-2})$ that colors $u$ and $v$ from $\{3,4\}$, also set $\alpha^{b-1}$ as the coloring in $V(G^{b-1})$ that agrees with $\beta^{b-2}$ on $u$ and $v$, and also set $\beta^{b-1}$ as the coloring in $V(G^{b-1})$ that agrees with $\pi$ on $u$ and $v$.  If $\beta^b$ uses both $3$ and $4$ on $\{u,v\}$, then set $\beta^{b-2}$ as a common neighbor in $G^{b-2}$ of $\alpha^{b-2}$ and a coloring $\rho\in V(G^{b-2})$ satisfying $\rho(u)=2$ and $\rho(v)\in\{3,4\}$, also set $\alpha^{b-1}$ as the coloring in $V(G^{b-1})$ that agrees on $u$ and $v$ with $\beta^{b-2}$, and set $\beta^{b-1}$ as the coloring in $V(G^{b-1})$ that agrees with $\rho$ on $u$ and $v$.  

We complete our Hamiltonian cycle through $G^1_4(H)$ by first taking our path through $G^{b-2}$ and $G^{b-1}$, then setting $\alpha^b$ as the coloring in $V(G^b)$ that agrees $\beta^{b-1}$ on $u$ and $v$ (such an $\alpha^b$ exists because $\phi^b(x)=1$ and $\phi^b(y)=2$ while $\beta^{b-1}$ colors $u$ from $\{2,3,4\}$ and colors $v$ from $\{3,4\}$, and $\beta^{b-1}$ and $\alpha^b$ are adjacent in $G^1_4(H)$ because they only differ on the vertex of $H'$ where $\phi^{b-1}$ and $\phi^b$ differ), and finally finding a Hamiltonian path through $G^b$ (such a path exists: if $\beta^b$ uses a color outside of $\{3,4\}$ on $u$ or $v$, then we selected $\alpha^{b-1}$ to color $u$ and $v$ from $\{3,4\}$, so there exists a Hamiltonian path through $G^b$ from $\alpha^{b-1}$ to any other vertex; if $\beta^b$ colors $u$ and $v$ from $\{3,4\}$, then we selected $\alpha^{b-1}$ to color $u$ with $2$ and $v$ from $\{3,4\}$, so there exists a Hamiltonian path through $G^b$ from $\alpha^{b-1}$ to any coloring that colors $u$ and $v$ from $\{3,4\}$).
\begin{case}
We have $m=b-2$ as well as $\phi^{b-1}(x)=\phi^b(x)=1$ and $\phi^{b-1}(y)=\phi^b(y)=2$.
\end{case}
Note that $G^{b-1}$ and $G^b$ both look like $F^2$ from Figure \ref{adjexthfourfig}.  Notice that if $\pi$ is one of the vertices of $F^2$ labeled $1212$, $1342$, or $1432$, and $\rho$ is any vertex of $F^2$ besides $\pi$, then there is a Hamiltonian path through $F^2$ whose endpoints are $\pi$ and $\rho$; pick $\pi$ to be any element of $\{1212,1342,1432\}$ that disagrees on $uv$ with both $\alpha^{b-1}$ and $\beta^b$.  When traversing the extensions of $\phi^{b-1}$, take the Hamiltonian path through $G^{b-1}$ from $\alpha^{b-1}$ to the coloring $\beta^{b-1}$ corresponding to $\pi$, and when traversing the extensions of $\phi^b$, take the Hamiltonian path through $G^{b-1}$ from the coloring $\alpha^b$ corresponding to $\pi$ to $\beta^b$.  This completes a Hamiltonian cycle through $G^1_4(H)$ because $\beta^{b-1}$ and $\alpha^b$ only disagree on the vertex where $\phi^{b-1}$ and $\phi^b$ disagree, so they are adjacent in $G^1_4(H)$.
\end{proof}
\setcounter{case}{0}
\begin{cor}
If $H$ is obtained from a loopless multigraph $M$ by subdividing each edge of $M$ at least twice, then $h_4(H)=1$.
\end{cor}
\begin{proof}
Let $H_1,\ldots,H_m$ be a sequence of subgraphs of $H$ described in Observation \ref{subobs}; since $M$ is loopless, for each $i\in[m-1]$ we have $H_i\subset^{x_i,y_i}_{\ell}H_{i+1}$ where $d_{H_i}(x_i,y_i)\geq 3$.  We have $h_4(H_1)=1$ since $H_1$ is a forest, and for $i\in [m-1]$, if $h_4(H_i)=1$, then $h_4(H_{i+1})=1$, by Proposition \ref{h4subdivide}.  Hence $h_4(H)=1$.
\end{proof}
Let $H'$ be a $3$-colorable subgraph of a graph $H$ such that $H-V(H')$ consists of an edge $uwv$, with $u$ having a single neighbor $x\in V(H')$ and $v$ having a single neighbor $y\in V(H')-N(x)$.  For proper $3$-colorings $\psi^1$ and $\psi^2$ of $H'$ satisfying $\psi^1(x)=\psi^2(x)=1$ and $\psi^i(y)=i$, Figure \ref{adjexththreefig} shows each subgraph $F^i$ of $G^2_3(H)$ induced by the set of proper $3$-colorings $\psi^i_1,\psi^i_2,\ldots$ of $H$ that agree with $\psi^i$ on $H'$, with node $\psi^i_{\ell}$ of $F^i$ labeled $\psi^i_{\ell}(x)\psi^i_{\ell}(u)\psi^i_{\ell}(w)\psi^i_{\ell}(v)\psi^i_{\ell}(y)$.  Note that if $\pi$ is one of the vertices of $F^1$ such that $\pi(u)=\pi(v)$, and $\rho$ is any vertex of $F^2$ besides $\pi$, then there is a Hamiltonian path through $F^2$ whose endpoints are $\pi$ and $\rho$; if instead $\pi$ is in $\{12131,13121\}$ but $\rho$ is not, then again there is a Hamiltonian path through $F^2$ whose endpoints are $\pi$ and $\rho$.  Also note that if $\pi$ is one of the vertices of $F^2$ labeled $12312$, $13132$, or $13232$, and $\rho$ is any vertex of $F^2$ besides $\pi$, then there is a Hamiltonian path through $F^2$ whose endpoints are $\pi$ and $\rho$; if instead $\pi$ is in $\{12132,13212\}$ but $\rho$ is not, then again there is a Hamiltonian path through $F^2$ whose endpoints are $\pi$ and $\rho$.
\begin{figure}[h]
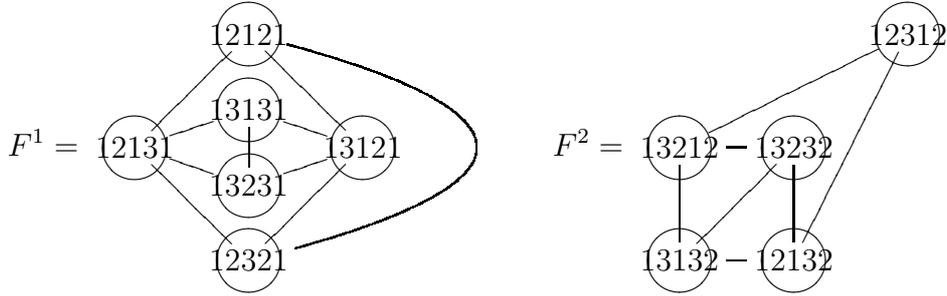

\centering
\[
F^1=
\xygraph{ !{<0cm,0cm>;<.5cm,0cm>:<0cm,.5cm>::}
!{(0,3) }*+[o][F-:<3pt>]{12121}="a"
!{(0,-3) }*+[o][F-:<3pt>]{12321}="b"
!{(0,1) }*+[o][F-:<3pt>]{13131}="c"
!{(3,0) }*+[o][F-:<3pt>]{13121}="d"
!{(0,-1) }*+[o][F-:<3pt>]{13231}="e"
!{(-3,0) }*+[o][F-:<3pt>]{12131}="f"
"a"-"d"-"b"-"f"-"a"-@/^3cm/"b" "c"-"d"-"e"-"f"-"c"-"e"
}
\hspace{1cm}
F^2=
\xygraph{ !{<0cm,0cm>;<.5cm,0cm>:<0cm,.5cm>::}
!{(3,3) }*+[o][F-:<3pt>]{12312}="a"
!{(-3,0) }*+[o][F-:<3pt>]{13212}="b"
!{(0,0) }*+[o][F-:<3pt>]{13232}="c"
!{(0,-3) }*+[o][F-:<3pt>]{12132}="d"
!{(-3,-3) }*+[o][F-:<3pt>]{13132}="e"
"a"-"b"-"c"-"d"-"e"-"b" "a"-"d" "c"-"e"
}\]
\caption{Two induced subgraphs of $G^2_3(H)$.}
\label{adjexththreefig}
\end{figure}
\begin{lem}\label{adjexth3}
Let $H'$ be a $3$-colorable subgraph of a graph $H$ such that $H-V(H')$ consists of an edge $uwv$, with $u$ having a single neighbor $x\in V(H')$ and $v$ having a single neighbor $y\in V(H')-N(x)$, and let $\phi$ and $\phi'$ be proper $3$-colorings of $H'$ adjacent in $G^2_3(H')$.  Letting $G$ denote the subgraph of $G^2_3(H)$ induced by the proper $3$-colorings of $H$ that agree on $H'$ with $\phi$, for every $\pi\in V(G)$ there exists $\rho\in V(G)-\{\pi\}$ such that there is a Hamiltonian path through $G$ from $\pi$ to $\rho$, and $\rho$ is adjacent in $G^2_3(H)$ to some proper $3$-coloring of $H$ that agrees with $\phi'$ on $H'$.
\end{lem}
\begin{proof}
Since $\phi$ and $\phi'$ are adjacent in $G^2_3(H')$, they differ on either one vertex or adjacent vertices of $H'$.  Since $x$ and $y$ are nonadjacent in $H'$ but $\phi$ and $\phi'$ are adjacent in $G^2_3(H')$, we either have $x=y$, or at least one of $\phi(x)=\phi'(x)$ or $\phi(y)=\phi'(y)$; in the former case, we assume without loss of generality that $\phi(x)=1$, and in the latter case, we assume without loss of generality that $\phi(x)=\phi'(x)=1$.  Let $\pi\in V(G)$; we find $\rho\in V(G)-\{q\}$ such that there is a Hamiltonian path through $G$ from $\pi$ to $\rho$, with $\rho(u)\neq \phi'(x)$ and $\rho(v)\neq \phi'(y)$ (allowing $\phi'$ to be extended to some proper $k$-coloring $\rho'$ of $H$ by coloring $uv$ like $\rho$, so $\rho$ and $\rho'$ will be adjacent in $G^2_3(H)$ since they only differ where $\phi$ and $\phi'$ differ).

First suppose $\phi(y)=1$, in which case $G$ looks like $F^1$ from Figure \ref{adjexththreefig}.  Either $\phi'(y)=1$ or $\phi'(y)\neq 1$, in which case without loss of generality assume $\phi'(y)=2$.  Thus we have $\phi(x)=\phi(y)=1$ as well as either $\phi'(x)=\phi'(y)\in[2]$, or $\phi'(x)=1$ and $\phi'(y)=2$.  In either case, there are extensions of both $\phi$ and $\phi'$ to $H$ that label $uwv$ as $313$ and $323$; for each $\pi\in V(G)$ there is a Hamiltonian path through $G$ from $\pi$ to at least one of these vertices, which we set as $\rho$.

Now suppose $\phi(y)\neq 1$ (so $x\neq y$, and $\phi'(x)=1$ by assumption); without loss of generality assume $\phi(y)=2$, in which case $G$ looks like $F^2$ from Figure \ref{adjexththreefig}.  If $\phi'(y)\in[2]$, then there are extensions of both $\phi$ and $\phi'$ that label $uwv$ as $313$ and $323$; for each $\pi\in V(G)$ there is a Hamiltonian path through $G$ from $\pi$ to at least one of these vertices, which we set as $\rho$.  If $\phi'(y)=3$, then there are extensions of both $\phi$ and $\phi'$ that label $uwv$ as $231$ and $321$; for each $\pi\in V(G)$ there is a Hamiltonian path through $G$ from $\pi$ to at least one of these vertices, which we set as $\rho$.
\end{proof}
\begin{prop}\label{h3subdivide}
Let $H'$ be a $3$-colorable subgraph of a graph $H$ such that $H-V(H')$ consists of a path $uwv$, with $w$ having no neighbor in $H'$, $u$ having a single neighbor $x\in V(H')$, $v$ having a single neighbor $y\in V(H')-N(x)$, and there existing a vertex $z\in V(H')-N[x]\cup N[y]$.  If $h_3(H')\leq 2$, then $h_3(H)\leq 2$.
\end{prop}
\begin{proof}
Since $h_3(H')\leq 2$, there exists a Hamiltonian cycle $[\phi^1,\ldots,\phi^b]$ through $G^2_3(H')$.  There exists $i$ such that $\phi^i(z)\neq \phi^{i+1}(z)$, in which case $\phi^i(x)=\phi^{i+1}(x)$ and $\phi^i(y)=\phi^{i+1}(y)$ since neither $x$ nor $y$ is $z$ or is adjacent to $z$.  If there exists an $i$ such that $\phi^i(x)=\phi^i(y)=\phi^{i+1}(x)=\phi^{i+1}(y)$, then without loss of generality assume $\phi^{b-1}(x)=\phi^{b-1}(y)=\phi^b(x)=\phi^b(y)=1$.  If there exists no such $i$, then there must exist $\ell$ such that $\phi^{\ell}(x)=\phi^{\ell+1}(x)\neq \phi^{\ell}(y)=\phi^{\ell+1}(y)$; without loss of generality assume $\phi^{b-1}(x)=\phi^b(x)=1$ and $\phi^{b-1}(y)=\phi^b(y)=2$.  To complete the proof, we alter the Hamiltonian cycle through $G^2_3(H')$ into a Hamiltonian cycle through $G^2_3(H)$ such that the extensions of each $\phi^i$ appear consecutively, starting with $\alpha^i$ and ending with $\beta^i$, with $\beta^i$ agreeing with $\alpha^{i+1}$ on $u$, $v$, and $w$.

For each $i\in[b]$, let $G^i$ denote the subgraph of $G^2_3(H)$ induced by the extensions of $\phi^i$. By Lemma \ref{adjexth3}, for every $\pi\in V(G^i)$ there exists $\rho\in V(G)-\{\pi\}$ such that there is a Hamiltonian path through $G^i$ from $\pi$ to $\rho$, and $\rho$ is adjacent in $G^2_3(H)$ to some extension of $\phi^{i+1}$.  Set $\alpha^1$ as any coloring in $V(G^1)$ for which there exist distinct colorings $\pi$ and $\rho$ in $V(G^b)$ such that there is a Hamiltonian path through $G^b$ from $\pi$ to $\rho$, and $\rho$ is adjacent in $G^2_3(H)$ to $\alpha^1$.  Order the extensions of $\phi^1,\ldots,\phi^{b-2}$, plus $\alpha^{b-1}$, so that the extensions of each $\phi^i$ form a Hamiltonian path through $G^i$ from $\alpha^i$ to $\beta^i$, with $\beta^i$ adjacent in $G^2_3(H)$ to $\alpha^{i+1}$.  Let $\beta^b$ be the coloring in $V(G^b)$ adjacent in $G^2_3(H)$ to $\alpha^1$.

By assumption we have either $\phi^{b-1}(x)=\phi^{b-1}(y)=\phi^b(x)=\phi^b(y)=1$, or $\phi^{b-1}(x)=\phi^b(x)=1$ and $\phi^{b-1}(y)=\phi^b(y)=2$.  In the former case, $G^{b-1}$ and $G^b$ both look like $F^1$ from Figure \ref{adjexththreefig}.  Notice that if $\pi$ is one of the vertices of $F^1$ labeled $12121$, $12321$, $13131$, or $13231$, and $\rho$ is any vertex of $F^1$ besides $\pi$, then there is a Hamiltonian path through $F^1$ whose endpoints are $\pi$ and $\rho$; pick $\pi$ to be any element of $\{12121,12321,13131,13231\}$ that disagrees with $\alpha^{b-1}$ and $\beta^{b}$ on $u$, $w$, and $v$.  In the latter case, $G^{b-1}$ and $G^b$ both look like $F^2$ from Figure \ref{adjexththreefig}.  Notice that if $\pi$ is one of the vertices of $F^2$ labeled $12312$, $13132$, or $13232$, and $\rho$ is any vertex of $F^2$ besides $\pi$, then there is a Hamiltonian path through $F^2$ whose endpoints are $\pi$ and $\rho$; pick $\pi$ to be any element of $\{12312,13132,13232\}$ that disagrees with $\alpha^{b-1}$ and $\beta^{b}$ on $u$, $w$, and $v$.  In either case, we can traverse the extensions of $\phi^B{b-1}$ by taking the Hamiltonian path through $G^{b-1}$ from $\alpha^{b-1}$ to the coloring $\beta^{b-1}$ corresponding to $\pi$, and we can traverse the extensions of $\phi^b$ by taking the Hamiltonian path through $G^{b-1}$ from the coloring $\alpha^b$ corresponding to $\pi$ to $\beta^b$.  This completes a Hamiltonian cycle through $G^2_3(H)$ because $\beta^{b-1}$ and $\alpha^b$ only disagree on the vertex where $\phi^{b-1}$ and $\phi^b$ disagree, so they are adjacent in $G^2_3(H)$.
\end{proof}
\begin{cor}
If $H$ is obtained from a multigraph $M$ by subdividing each edge of $M$ at least three times, then $h_3(H)\leq 2$.
\end{cor}
\begin{proof}
Since $h_3(P_n)=1$ for $n\geq 5$ and $h_3(C_n)=2$ for $n\geq 4$, we may assume $M$ has more than one edge.  Let $H_1,\ldots,H_m$ be a sequence of subgraphs of $H$ described in Observation \ref{subobs}; for each $i\in[m-1]$ we have $H_i\subset^{x_i,y_i}_{\ell}H_{i+1}$ where either $x_i=y_i$ or $d_{H_i}(x_i,y_i)\geq 4$, and $V(H_i)-N[x^i]\cup N[y^i]\neq\emptyset$.  We have $h_3(H_1)\leq 2$ since $H_1$ is a forest, and for $i\in [m-1]$, if $h_3(H_i)\leq 2$, then $h_3(H_{i+1})\leq 2$, by Proposition \ref{h3subdivide}.  Hence $h_3(H)\leq 2$.
\end{proof}

\section{Complete Multipartite Graphs}
\label{sec:multi}
In this section we prove Theorem \ref{multi}, concerning complete multipartite graphs.  To prove our first result, we use the following theorem of Kompel'makher and Liskovets from 1975 \cite{KL}.  Given a set $T$ of transpositions acting on permutations of $[n]$, let $G(T)$ be the graph whose vertices are the elements of $[n]$, with edges joining $b$ and $c$ if and only if some transposition in $T$ swaps the values in positions $b$ and $c$; we call $T$ a \emph{basis of transpositions} if $G(T)$ is a tree.  If $T$ is a basis of transpositions, then the permutations of $[n]$ can be ordered cyclically so that consecutive permutations differ by a transposition in $T$.  Note that if $T$ consists of all transpositions involving the first position, then $G(T)$ is a star, so $T$ is a basis of transpositions.
\begin{thm}
If $H=K_{m_1,\ldots ,m_k}$, where $k\geq 2$ and $m_1\leq\cdots\leq m_k$, then $g_k(H)=h_k(H)=m_1+m_k$.
\end{thm}
\begin{proof}
Since $g_k(H)\leq h_k(H)$, it suffices to show $g_k(H)\geq m_1+m_k$ and $h_k(H)\leq m_1+m_k$.  Let the partite sets of $H$ be $X_1,\ldots ,X_k$, with $|X_i|=m_i$ for each $i\in [k]$.  The only proper $k$-colorings of $H$ assign the elements of $[k]$ to the partite sets $X_1,\ldots ,X_k$ in a one-to-one fashion, coloring each partite set monochromatically.  Thus the proper $k$-colorings of $H$ correspond in a one-to-one fashion with the proper $k$-colorings of $K_k$.

For colorings differing on $X_k$ to be in the same component of $G^j_k(H)$, there must be adjacent vertices in $G^j_k(H)$ that differ on $X_k$ and some other partite set.  Since $X_1$ is the smallest partite set, $g_k(H)\geq m_1+m_k$.

By \cite{KL}, there is a cyclic ordering $C$ of the permutations of $[k]$ such that consecutive permutations differ in the first position and exactly one other position.  When $j\geq m_1+m_k$, the ordering $C$ corresponds to a Hamiltonian cycle through $G_k^j(H)$, since successive steps are performed by interchanging the colors on the smallest partite set and one other partite set.  Hence $h_k(H)\leq m_1+m_k$.
\end{proof}
\begin{thm}
If $H$ is a complete $k$-partite graph and $\ell>k$, then $g_{\ell}(H)=1$.
\end{thm}
\begin{proof}
We prove the theorem by first showing that any proper $\ell$-coloring of $H$ is in the same component of $G^1_{\ell}(H)$ as some proper $k$-coloring of $H$, then showing that all proper $k$-colorings of $H$ are in the same component of $G^1_{\ell}(H)$.  For the first claim, if $\phi$ is a proper $\ell$-coloring of $H$, then $\phi$ assigns no color to multiple partite sets, so each partite set $X$ can be recolored monochromatically to some color assigned by $\phi$ to one of its vertices.  For the second claim, suppose $\phi$ only uses a set $S$ of $k$ colors, and note that the color $b$ given to any partite set $X$ could be changed one vertex at a time to any color $c\notin S$.  If $X$ is to be recolored with some color $d$ already assigned to some partite set $Y$, then recolor $Y$ with $c$ before recoloring $X$ with $d$.  Since no proper coloring gives the same color to multiple partite sets, this process can be applied to each partite set until the desired coloring is obtained.
\end{proof}
Given distinct colors $b$ and $c$, let $Q_n(b,c)$ be the $n$-dimensional hypercube with a vertex for each $n$-bit binary string from the alphabet $\{b,c\}$ and an edge between vertices differing in exactly one coordinate.  As in \cite{CM}, we shall use the well-known facts that $Q_n(b,c)$ is Hamiltonian for all $n\geq 2$, and $Q_n(b,c)$ contains a Hamiltonian path from $b\cdots b$ to $c\cdots c$ if and only if $n$ is odd.
\begin{thm}
Let $H$ be a complete $k$-partite graph.  Then $h_{k+1}(H)=1$ if each partite set has an odd number of vertices, and $h_{k+1}(H)=2$ otherwise.
\end{thm}
\begin{proof}
Let $H$ have partite sets $X_1,\ldots ,X_k$, and let $K_k$ have vertex set $[k]$.  Set $n=(k+1)!$.  Since $h_{k+1}(K_k)=1$, there exists a Hamiltonian cycle $[\phi_1,\ldots ,\phi_n]$ through $G^1_{k+1}(K_k)$.  For $i\in [n]$, let $a_i$ be the vertex of $K_k$ that receives different colors from $\phi_i$ and $\phi_{i+1}$, with $\phi_i(a_i)=b_i$ and $\phi_{i+1}(a_i)=c_i$; note that $a_i\neq a_{i+1}$ (if $a_i=a_{i+1}$, then we would have $\phi_{i+2}=\phi_i$ if $c_{i+1}=b_i$, and $\phi_{i+2}=\phi_{i+1}$ if $c_{i+1}=c_i$, with $\phi_{i+2}$ using color $c_{i+1}$ on both $a_i$ and some neighbor of $a_i$ if $c_{i+1}\in [k+1]-\{b_i,c_i\}$).  If $R$ is a path $\alpha_1,\ldots ,\alpha_m$ in $G^1_2(X_{a_i})$ such that each $\alpha_{\ell}$ colors the vertices of $X_{a_i}$ using colors $b_i$ and $c_i$, then let $\phi_i\cdot R$ denote the path $\pi_1,\ldots,\pi_m$ in $G^1_{k+1}(H)$ such that $\pi_{\ell}(v)=\alpha_{\ell}(v)$ if $v\in X_{a_i}$, and $\pi_{\ell}(v)=\phi_i(d)$ if $v\in X_d$ for $d\in [k]-\{a_i\}$.  Indeed, $\phi_i\cdot R$ is a path in $G^1_{k+1}(H)$ because $\pi_{\ell}$ and $\pi_{\ell+1}$ differ only on the vertex of $X_{a_i}$ where $\alpha_{\ell}$ and $\alpha_{\ell+1}$ differ.

For $i\in[n]$, view each vertex of the hypercube $Q_{|X_{a_i}|}(b_i,c_i)$ as a coloring of $X_{a_i}$ using the colors $b_i$ and $c_i$ (so the $j$th vertex of $X_{a_i}$ is colored according to the $j$th coordinate of the given hypercube vertex).  Hence paths in $Q_{|X_{a_i}|}(b_i,c_i)$ correspond to paths in $G^1_2(X_{a_i})$, since adjacent vertices $\alpha$ and $\beta$ in $Q_{|X_{a_i}|}(b_i,c_i)$ differ in exactly one coordinate, which is the only vertex of $X_{a_i}$ on which the colorings of $X_{a_i}$ corresponding to $\alpha$ and $\beta$ differ. 

We are now ready to prove the theorem via three claims. 
\begin{claim}
We have $h_{k+1}(H)\leq 2$.
\end{claim}
There exists a Hamiltonian cycle through $Q_{|X_{a_i}|}(b_i,c_i)$ for each $i\in [n]$; break that cycle up into two directed paths $R_i$ and $S_i$, with $R_i$ starting at $b_i\cdots b_i$ and $S_i$ starting at $c_i\cdots c_i$.  Note that the other endpoint of $R_i$ uses $b_i$ exactly once, and the other endpoint of $S_i$ uses $c_i$ exactly once.  Let $S'_i$ be $S_i$ with $c_i\cdots c_i$ deleted, so $S'_i$ starts by using $b_i$ exactly once.  To prove the claim, we show that $[\phi_1\cdot R_1,\ldots ,\phi_n\cdot R_n,\phi_1\cdot S'_1,\ldots ,\phi_n\cdot S'_n]$ is a Hamiltonian cycle through $G^2_{k+1}(H)$:
\begin{itemize}
\item
Every proper $(k+1)$-coloring $\phi$ of $H$ is included exactly once: the proper $(k+1)$-colorings of $H$ that use only $k$ colors correspond to the proper $(k+1)$-colorings of $K_k$ (since no color can appear in multiple partite sets), which in turn correspond to the initial colorings of $\phi_i\cdot R_i$ for $i\in [n]$.  The proper $(k+1)$-colorings of $H$ that use all $k+1$ colors can be uniquely obtained from our Hamiltonian cycle $[\phi_1,\ldots ,\phi_n]$ through $G^1_{k+1}(K_k)$ by coloring $X_{a_i}$ using both $b_i$ and $c_i$ (the ways of doing which correspond to the vertices of $Q_{|X_{a_i}|}(b_i,c_i)$ besides $b_i\cdots b_i$ and $c_i\cdots c_i$) while coloring $X_d$ monochromatically with $\phi_i(d)$ for each $d\neq a_i$; thus these colorings of $H$ correspond to those in $\phi_i\cdot R_i$ or $\phi_i\cdot S'_i$ for $i\in [n]$, minus the initial colorings of $\phi_i\cdot R_i$.
\item
For $i\in[n]$, $\phi_i\cdot R_i$ and $\phi_i\cdot S'_i$ are paths in $G^1_{k+1}(H)$. 
\item
For $i\in [n-1]$, the last coloring of $\phi_i\cdot R_i$ is adjacent in $G^1_{k+1}(H)$ to the first coloring of $\phi_{i+1}\cdot R_{i+1}$ because they differ only on the lone vertex in $X_{a_i}$ colored $b_i$ by the last vertex in $R_i$.
\item 
The last coloring of $\phi_n\cdot R_n$ is adjacent in $G^2_{k+1}(H)$ to the first coloring of $\phi_1\cdot S'_1$ because they differ only on the edge $uv$, where $u$ is the lone vertex in $X_{a_n}$ colored $b_n$ by the last vertex in $R_n$, and $v$ is the lone vertex in $X_1$ colored $b_1$ by the first vertex in $S'_1$ ($u$ and $v$ are adjacent because $a_n\neq a_1$). 
\item
For $i\in [n-1]$, the last coloring of $\phi_i\cdot S'_i$ is adjacent in $G^2_{k+1}(H)$ to the first coloring of $\phi_{i+1}\cdot S'_{i+1}$ because they differ only on the edge $uv$, where $u$ is the lone vertex in $X_{a_i}$ colored $c_i$ by the last vertex in $S'_i$, and $v$ is the lone vertex in $X_{a_{i+1}}$ colored $b_{i+1}$ by the first vertex in $S'_{i+1}$ ($u$ and $v$ are adjacent because $a_i\neq a_{i+1}$).  
\item
The last coloring of $\phi_n\cdot S'_n$ is adjacent in $G^1_{k+1}(H)$ to the first coloring of $\phi_1\cdot R_1$ because they differ only on the lone vertex in $X_{a_n}$ colored $c_n$ by the last vertex in $S'_n$.
\end{itemize}
\begin{claim}
If $|X_i|$ is odd for each $i\in [k]$, then $h_{k+1}(H)=1$.
\end{claim}
If each partite set of $H$ has an odd number of vertices, then there exists a Hamiltonian path $T_i$ from $b_i\cdots b_i$ to $c_i\cdots c_i$ in the hypercube $Q_{|X_{a_i}|}(b_i,c_i)$ for each $i\in [n]$; let $T'_i$ be $T_i$ with $c_i\cdots c_i$ deleted, so the last vertex of $T'_i$ uses $b_i$ exactly once.  To prove the claim, we show that $[\phi_1\cdot T'_1,\ldots ,\phi_n\cdot T'_n]$ is a Hamiltonian cycle through $G^1_{k+1}(H)$:
\begin{itemize}
\item
Every proper $(k+1)$-coloring $\phi$ of $H$ is included exactly once: the proper $(k+1)$-colorings of $H$ that use only $k$ colors correspond to the proper $(k+1)$-colorings of $K_k$, which in turn correspond to the initial colorings of $\phi_i\cdot T'_i$ for $i\in [n]$.  The proper $(k+1)$-colorings of $H$ that use all $k+1$ colors can be uniquely obtained from our Hamiltonian cycle $[\phi_1,\ldots ,\phi_n]$ through $G^1_{k+1}(K_k)$ by coloring $X_{a_i}$ using both $b_i$ and $c_i$ (the ways of doing which correspond to the vertices of $Q_{|X_{a_i}|}(b_i,c_i)$ besides $b_i\cdots b_i$ and $c_i\cdots c_i$) while coloring $X_d$ monochromatically with $\phi_i(d)$ for $d\neq a_i$; thus these colorings of $H$ correspond to those in $\phi_i\cdot T'_i$ for $i\in [n]$, minus the initial colorings of $\phi_i\cdot T'_i$.
\item
For $i\in[n]$, $\phi_i\cdot T'_i$ is a path in $G^1_{k+1}(H)$. 
\item
For $i\in [n]$, the last coloring of $\phi_i\cdot T'_i$ is adjacent in $G^1_{k+1}(H)$ to the first coloring of $\phi_{i+1}\cdot T'_{i+1}$ (letting $\phi_{n+1}=\phi_1$ and $T'_{n+1}=T'_1$) because they differ only on the lone vertex in $X_{a_i}$ colored $b_i$ by the last vertex in $T'_i$.
\end{itemize}
\begin{claim}
If $h_{k+1}(H)=1$, then $|X_i|$ is odd for each $i\in[k]$.
\end{claim}
Let $i\in [k]$.  Either $|X_i|=1$, or there exists a proper $(k+1)$-coloring $\phi$ of $H$ that uses distinct colors $b$ and $c$ on $X_i$.  Note that $\phi$ must color each vertex of $X_i$ with $b$ or $c$, and each partite set besides $X_i$ must receive exactly one color, which cannot appear elsewhere (there are $k-1$ partite sets besides $X_i$, and they must be colored with the $k-1$ colors of $[k+1]-\{b,c\}$ in order for $\phi$ to be a proper $(k+1)$-coloring of $H$).  If $\phi'$ is adjacent to $\phi$ in $G^1_{k+1}(H)$, then $\phi'$ must disagree with $\phi$ on $X_i$ and agree with $\phi$ outside of $X_i$ (if $\phi$ and $\phi'$ agreed on $X_i$, then they would have to disagree on multiple partite sets besides $X_i$, in which case they wouldn't be adjacent in $G^1_{k+1}(H)$).  Therefore, if $W$ is the set of $(k+1)$-colorings of $H$ that agree with $\phi$ outside of $X_i$, then there are only two colorings $\pi$ and $\alpha$ in $W$ that have neighbors in $G^1_{k+1}(H)$ outside of $W$: one colors $X_i$ monochromatically with $b$, and the other colors $X_i$ monochromatically with $c$.  Thus any Hamiltonian cycle through $G^1_{k+1}(H)$ must contain a $\pi,\alpha$-path $P$ whose vertices are the colorings agreeing with $\phi$ outside of $X_i$.  Hence the restriction of $P$ to $X_i$ yields a Hamiltonian path through the hypercube $Q_{|X_{a_i}|}(b,c)$ between $b\cdots b$ and $c\cdots c$, so $|X_i|$ must be odd.
\end{proof}

\section{Acknowledgements} 
\label{sec:acknowledgements}
The author acknowledges support from National Science Foundation grant DMS 08-38434\linebreak EMSW21-MCTP: Research Experience for Graduate Students.

\end{document}